\documentclass[10pt,letterpaper]{amsart} 

\usepackage{color}
\usepackage{multicol}
\usepackage[letterpaper, margin=1in]{geometry}
\usepackage{subfig}
\usepackage[overload]{empheq}

\usepackage{xcolor}

\usepackage{listings}
\usepackage{amsmath, amsthm, amssymb, wasysym, verbatim, bbm, graphics}
\usepackage{enumerate}
\usepackage{url}
\usepackage{mathtools}
\usepackage[colorlinks,linkcolor=blue]{hyperref}
\usepackage{romannum}
\usepackage{xfrac}
\usepackage{float}
\usepackage{amsmath,bm}
\usepackage[title]{appendix}
\usepackage{todonotes}
\usepackage{subcaption}
\usepackage{caption}

\newtheorem{lemma}{Lemma}[section]

\newtheorem*{maintheorem*}{Main Theorem}
\theoremstyle{definition}{}

\newcommand{\Ltwo}{L^2(\mathcal{T}_h)}
\newcommand{\LtwoK}{L^2(K)}
\newcommand{\Ltwovector}{\pmb{L}^2(\mathcal{T}_h)}
\newcommand{\Ltwoboundary}{L^2(\partial\Th)}

\newcommand{\Ltwoboundaryb}{L^2(\pThb)}

\newcommand{\Hone}{H^1(\mathcal{T}_h)}

\newcommand{\Rd}{\mathbb{R}^d}
\newcommand{\Th}{\mathcal{T}_h}
\newcommand{\Pk}{\mathcal{P}^k}
\newcommand{\Pone}{\mathcal{P}^1}
\newcommand{\Pzero}{\mathcal{P}^0}
\newcommand{\Pkvector}{\pmb{\mathcal{P}}^k}

\newcommand{\Fhk}{\mathcal{F}_h^k}
\newcommand{\Uhk}{\mathcal{U}_h^k}
\newcommand{\Xhk}{\mathcal{X}_h^k}
\newcommand{\Vhk}{\pmb{\mathcal{V}}_h^k}
\newcommand{\CRspace}{\mathcal{CR}(\Th)}

\newcommand{\pThi}{\partial\Th^i}
\newcommand{\pThb}{\partial\Th^b}

\newcommand{\uh}{u_h}
\newcommand{\uhath}{\hat{u}_h}
\newcommand{\uhaverageBoundary}{\bar{u}_{h,\Kboundary}}
\newcommand{\uhathaverageEdge}{\Bar{\hat{u}}_{h,e}}

\newcommand{\uhathaverageEdgei}{\Bar{\hat{u}}_{h,e_i}}
\newcommand{\uhaverageEdgei}{\Bar{u}_{h,e_i}}
\newcommand{\uhataverage}{\Bar{\hat{u}}_h}
\newcommand{\uhaverageK}{\Bar{u}_{h,K}}
\newcommand{\omegahaverageK}{\Bar{\omega}_{h,K}}
\newcommand{\ph}{\pmb{p}_h}
\newcommand{\qh}{\pmb{q}_h}
\newcommand{\vh}{v_h}

\newcommand{\gradienth}{\nabla_h}

\newcommand{\gradientKlifting}{{\mathcal{G}}^{\partial K}_h}
\newcommand{\gradientlifting}{{\mathcal{G}}^{k}_h}
\newcommand{\CRlifting}{{\mathcal{L}}^{\mathcal{CR}}_h}

\newcommand{\Kboundary}{\partial K}

\newtheorem{thm}[lemma]{Theorem}

\newtheorem{prop}[lemma]{Proposition}
\newtheorem{rem}[lemma]{Remark}
\newtheorem{lem}[lemma]{Lemma}
\newtheorem{cor}[lemma]{Corollary}

\allowdisplaybreaks

\numberwithin{equation}{section}

\title[Discrete Poincaré and Trace Inequalities for HDG]{Discrete Poincar\'e and Trace Inequalities for the Hybridizable Discontinuous Galerkin Method
}
\date{\today}

\author[Y. Yue]{Yukun Yue}
\address[Yukun Yue]{ \newline University of Wisconsin Madison, Madison, USA.}
\email{yyue@math.wisc.edu}

\begin{document}
 \pagenumbering{arabic}
\maketitle

\begin{abstract}
In this paper, we derive discrete Poincar\'e and trace inequalities for the hybridizable discontinuous Galerkin (HDG) method. We employ the Crouzeix-Raviart space as a bridge, connecting classical discrete functional tools from Brenner’s foundational work \cite{brenner2003poincare} with hybridizable finite element spaces comprised of piecewise polynomial functions defined both within the element interiors and on the mesh skeleton. This approach yields custom-tailored inequalities that underpin the stability analysis of HDG discretizations. The resulting framework is then used to demonstrate the well-posedness and robustness of HDG-based numerical schemes for second-order elliptic problems, even under minimal regularity assumptions on the source term and boundary data.

\end{abstract}

\section{Introduction}

We take $\Omega$ as a connected, bounded, open polyhedral domain within $\mathbb{R}^d$, where $d$ is either 2 or 3. $H^1(\Omega)$ denotes the standard Sobolev space consisting of functions in $L^2(\Omega)$ (the set of square integrable functions over $\Omega$) whose first-order distributional derivatives are also in $L^2(\Omega)$. The classical Poincaré-Friedrichs inequalities for $H^1$ functions are outlined as follows \cite{nevcas1967methodes, wloka1987partial}:

\begin{equation}
    \label{eq:classical_Poincare1}
    \|f\|^2_{L^2(\Omega)}\lesssim \| \nabla f\|^2_{L^2(\Omega)}+\left(\int_\Omega f\,dx \right)^2\quad\quad \quad\quad \forall f\in H^1(\Omega),
\end{equation}
and
\begin{equation}
    \label{eq:classical_Poincare2}
    \|f\|^2_{L^2(\Omega)}\lesssim \| \nabla f\|^2_{L^2(\Omega)}+\left(\int_{\Gamma} f\,ds \right)^2\quad\quad \quad\quad \forall f\in H^1(\Omega).
\end{equation}
where $\Gamma$ represents a measurable subset of $\partial\Omega$ with a positive $(d-1)$-dimensional measure. Additionally,
the classical trace inequality is presented as follows \cite{ding1996proof,evans2022partial}:

\begin{equation}
    \label{eq:classical:trace}
    \|f\|_{L^2(\partial \Omega)}^2\lesssim \|f\|_{H^1(\Omega)}^2.
\end{equation}
Applying the Poincaré-Friedrichs inequalities from \eqref{eq:classical_Poincare1} to \eqref{eq:classical_Poincare2}, the trace inequality is thus reformulated as:
\begin{equation}
    \label{eq:classical_trace1}
    \|f\|^2_{L^2(\partial\Omega)}\lesssim \| \nabla f\|^2_{L^2(\Omega)}+\left(\int_\Omega f\,dx \right)^2\quad\quad \quad\quad \forall f\in H^1(\Omega),
\end{equation}
and
\begin{equation}
    \label{eq:classical_trace2}
    \|f\|^2_{L^2(\partial\Omega)}\lesssim \| \nabla f\|^2_{L^2(\Omega)}+\left(\int_{\Gamma} f\,ds \right)^2\quad\quad \quad\quad \forall f\in H^1(\Omega).
\end{equation}
Our primary focus is on extending the classical inequalities established in the $H^1$ space to formulations that are specifically tailored for the analysis of hybridizable discontinuous Galerkin (HDG) methods \cite{chen2023superconvergence,chen2019hdg,cockburn2010hybridizable,cockburn2017note,cockburn2009unified,cockburn2019interpolatory,du2020new,fu2015analysis,nguyen2010hybridizable,qiu2018hdg}. These methods are formulated on piecewise polynomial hybridizable spaces.

The key distinction between hybridizable spaces and other more traditional nonconforming spaces \cite{altmann2012p,brenner2015forty,Cai2000,crouzeix1973conforming}, often utilized in the discontinuous Galerkin (DG) method, is depicted in Figure \ref{fig:DG-HDG-space}. In the DG method, the approach involves using piecewise polynomial functions; for instance, as illustrated in the left part of Figure \ref{fig:DG-HDG-space}, distinct functions $u_K$ and $u_L$ are selected for elements $K$ and $L$ respectively, and these functions may be discontinuous across the boundary where the two elements meet. Conversely, within hybridizable spaces, the function values are not only determined within the elements' interiors but also the function values on the boundaries are treated as separate variables. This idea is represented in the right part of Figure \ref{fig:DG-HDG-space}, where $u_K$ and $u_L$ are defined within elements $K$ and $L$ respectively, and $\hat{u}_{e_{K,L}}$ is defined on the interface $e_{K,L}$ between the two elements.

\begin{figure}
    \centering
    \includegraphics[width=5in]{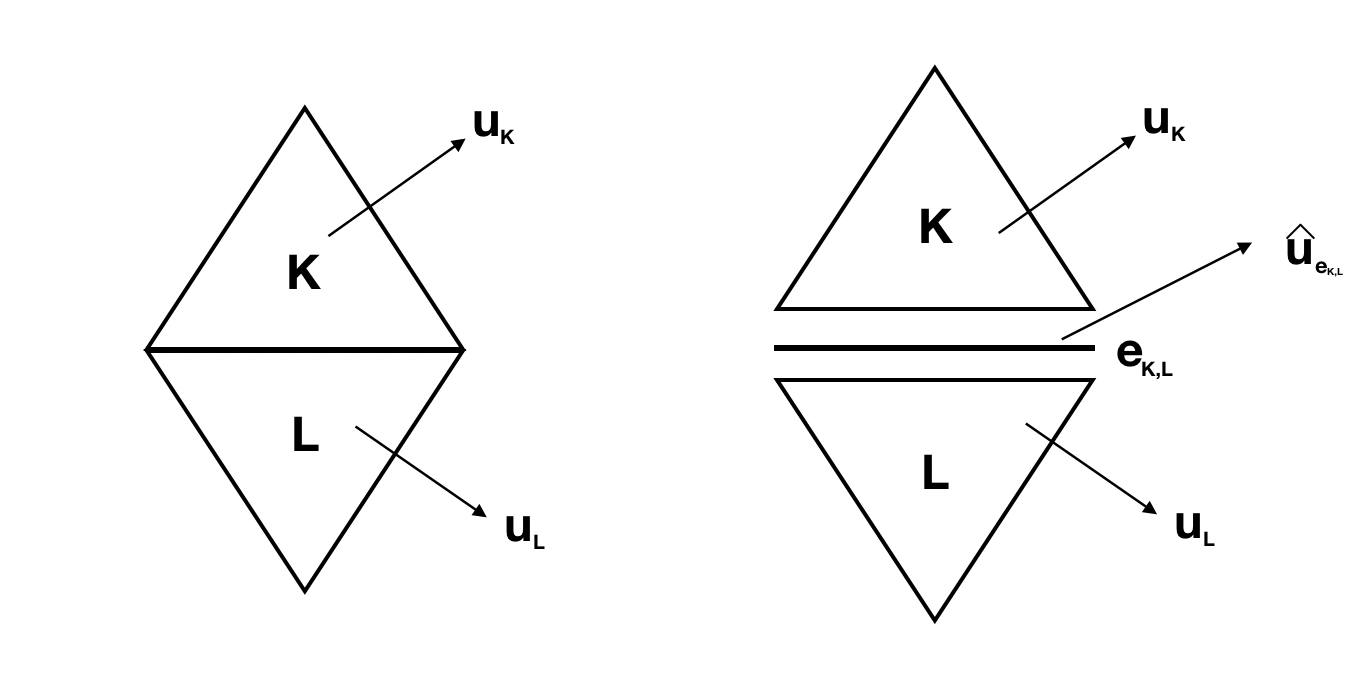}
    \caption{Comparison of Spaces: Nonconforming Space for DG Method (Left) versus Hybridizable Space for HDG or HHO Method (Right)}
    \label{fig:DG-HDG-space}
\end{figure}

In recent years, significant efforts have been made to develop analogs of Poincar\'e-Friedrichs and trace inequalities as analytical tools for nonconforming spaces, which are extensively used in the analysis of various numerical methods. For further information, we direct readers to \cite{brenner2015piecewise,brenner2003poincare,buffa2009compact,cockburn2018discrete,eymard2000finite,vohralik2005discrete,warburton2003constants} and references therein. Notably, \cite{brenner2003poincare} established a foundational discrete Poincaré inequality for piecewise $H^1$ functions, applicable to a wide range of nonconforming spaces used in DG methods. {Additionally, \cite{di2020hybrid} had a comprehensive review of functional tools, including Poincaré and trace inequalities, to analyze hybrid-type methods. Further results on Poincaré inequalities in various settings can be found in \cite[Proposition 5.4]{di2017hybrid}, \cite[Theorem 6.1]{di2010discrete}, \cite[Lemma B.14]{droniou2018gradient}, and \cite[Lemma 5.2]{eymard2010discretization}.}

{However, the direct application of these findings is insufficient to establish numerical stability for HDG methods due to challenges in controlling the jump term. Specifically, all these estimates rely on a term involving the norm of the difference between the interior variable and the trace variable, which is defined in a hybrid method and can be interpreted as a jump term, with a coefficient of order $O(\frac{1}{h})$. Here, we use the mesh-specific term $h_K$ to denote the diameter of a simplex $K$ in the mesh and take $h=\max_K h_K$, appearing in both the Poincaré inequality and the trace inequality. This poses a challenge, as in many cases, the best possible estimate for this norm obtained from the discrete formulation is only of order $O(1)$ which cannot control a term of order $O(\frac{1}{h})$.} 

{The main contribution of this work is the derivation of a discrete Poincaré inequality—with a jump term of order $O(h)$—and a discrete trace inequality of order $O(1)$ within a mixed formulation framework. A key insight is that the Crouzeix-Raviart space can be leveraged as a bridge between the HDG space and the classical discrete inequalities \cite{brenner2015piecewise,brenner2003poincare}, effectively absorbing the additional high-order jump term into the norm of the auxiliary variable introduced by the method. These refined estimates form the foundation for establishing the stability of the HDG method under minimal regularity assumptions. A detailed analysis is provided in Section \ref{subsec:discrete_poincare_brenner}, and the applications of these results to HDG methods are discussed in Section \ref{sec:HDG}.
}

We begin this work by deriving analogues of the classical Poincaré inequalities \eqref{eq:classical_Poincare1}–\eqref{eq:classical_Poincare2} for any pair of piecewise polynomial functions $(\uh,\uhath)$ in the hybridizable space $\Xhk$ (a detailed definition of $\Xhk$ is provided in Section \ref{subsec:notation}):

\begin{equation}\label{eq:poincare_mean1_intro}
          \|\uh\|^2_{\Ltwo}\lesssim {h}^2\lvert\uh \rvert_{\Hone}^2 +{h} \| \uh-\uhath\|^2_{L^2(\partial\Th)}+ \left| \CRlifting(\uhataverage) \right|_{\Hone}^2 + \left(\int_\Omega\uh\,dx \right)^2
    \end{equation}
    and
    \begin{equation}\label{eq:poincare_boundary_mean_intro}
          \|{\uh}\|^2_{\Ltwo}\lesssim {h}^2\lvert\uh \rvert_{\Hone}^2 + {h} \| \uh-\uhath\|^2_{L^2(\partial\Th)}+ \left| \CRlifting(\uhataverage) \right|_{\Hone}^2 + \left(\int_\Gamma\uh\,ds \right)^2.
    \end{equation}
In this context, $\uh$ represents a function defined within the interior of each element, whereas $\uhath$ pertains to a function specified on the mesh's skeleton.  The term $\CRlifting$ refers to a lifting operator that maps piecewise constant functions into piecewise linear functions within the Crouzeix-Raviart space, detailed further in Section \ref{subsec:CRLifting}. $\uhataverage$ is a piecewise constant function representing the average of $\uhath$ across each face of the mesh. Additional details on this are provided in Section \ref{sec:poincare}, while a comprehensive explanation of all other relevant notations is provided in Section \ref{subsec:notation}.

Moreover, we develop analogues of the trace inequalities, from \eqref{eq:classical_trace1} to \eqref{eq:classical_trace2}, for any pair of piecewise polynomial functions $(\uh,\uhath)\in\Xhk$:

 \begin{equation}\label{eq:hybrid_trace_u_intro}
         \|\uh\|_{\Ltwoboundaryb}^2\lesssim {h}\lvert \uh\rvert_{\Hone}^2+\|\uh-\uhath\|^2_{L^2(\pThb)}+(1+{h})\lvert \CRlifting(\uhataverage) \rvert_{\Hone}^2 + \left(\int_{\Gamma} \uh\,ds\right)^2
    \end{equation}
    and 
    \begin{equation}\label{eq:hybrid_trace_uhat_intro}
         \|\uhath\|_{\Ltwoboundaryb}^2\lesssim {h}\lvert \uh\rvert_{\Hone}^2+\|\uh-\uhath\|^2_{L^2(\pThb)}+(1+{h})\lvert \CRlifting(\uhataverage) \rvert_{\Hone}^2 + \left(\int_{\Gamma} \uh\,ds\right)^2.
    \end{equation}

Our proof of both the Poincaré and trace inequalities relies on the work presented in \cite{brenner2015piecewise} and \cite{brenner2003poincare}. By employing a Crouzeix-Raviart lifting, we bridge hybridizable spaces with traditional nonconforming spaces. This methodology facilitates the use of theories from \cite{brenner2015piecewise} and \cite{brenner2003poincare}, allowing us to eliminate jump terms. Consequently, this enables us to achieve a novel estimate.

  Another main contribution of this paper involves utilizing the Poincaré inequality, as detailed from \eqref{eq:poincare_mean1_intro} to \eqref{eq:poincare_boundary_mean_intro}, and the trace inequality, from \eqref{eq:hybrid_trace_u_intro} to \eqref{eq:hybrid_trace_uhat_intro}, to prove the stability of second-order elliptic equations solved by the HDG formulation. This approach represents a variation on the proof technique found in \cite{jiang2023stability}, which relies on a translation argument stemming from \eqref{eq:poincare_boundary_mean_intro} for establishing stability. In Section \ref{sec:HDG}, we will provide a proof based on the mathematical tools developed in this work directly.

    {Specifically, introducing an auxiliary variable
$$\ph = -a\nabla\uh,$$
where $a(x)$ is a matrix-valued function that is symmetric and uniformly positive definite on $\Omega$, is a common practice in the standard HDG formulation when a mixed method (see, e.g., \cite{arnold1985mixed,arnold1990mixed,arnold2002mixed}) is considered for second-order elliptic equations \cite{cockburn2004characterization,cockburn2009unified,cockburn2009superconvergent,sevilla2016tutorial} in order to develop the numerical scheme.
} For each element $K$, given $(\uh,\uhath)\in \Xhk$, one can determine $\pmb{p}_h\in \Vhk$ as it satisfies the following relation:
\begin{equation}
    \label{eq:HDG_formulation_p_nabla_u}
    (\ph,\qh)_K = (\uh,\nabla\cdot\qh)_K-\langle \uhath,\qh\cdot\pmb{n}  \rangle _{\Kboundary}
\end{equation}
where $\Vhk$ is space of piecewise-polynomial vector-valued functions that will be defined in Section \ref{subsec:notation}. Introducing such a variable typically leads to an energy term involving $\|\ph\|_{\Ltwovector}^2$, as opposed to $\|\nabla u\|_{\Ltwovector}^2$, which is more commonly encountered in classical elliptic theory \cite{evans2022partial}. The specifics of this distinction will be elaborated in Section \ref{sec:HDG}. As a result, we require a variation of the existing analytical tools, adapted for use with $\ph$. In the last section of this paper, we introduce the following findings: The Poincaré inequality from \eqref{eq:poincare_mean1_intro} to \eqref{eq:poincare_boundary_mean_intro} is revised as

  \begin{equation}\label{eq:poincare_mean2_ph_intro}
          \|\uh\|^2_{\Ltwo}\lesssim (1+{h}^2)\|\ph\|_{\Ltwovector}^2 +{h} \| \uh-\uhath\|^2_{L^2(\partial\Th)}+ \left(\int_\Omega\uh\,dx \right)^2
    \end{equation}
    and
\begin{equation}\label{eq:poincare_boundary_mean_ph}
          \|\uh\|^2_{\Ltwo}\lesssim \left(1+{h}^2\right)\|\ph\|_{\Ltwovector}^2+{h} \| \uh-\uhath\|^2_{L^2(\partial\Th)} + \left(\int_\Gamma\uhath\,ds \right)^2.
    \end{equation}
The trace inequalities \eqref{eq:hybrid_trace_u_intro}-\eqref{eq:hybrid_trace_uhat_intro} can be written as
\begin{equation}\label{eq:hybrid_trace_u_ph_intro}
        \|\uh\|_{{\Ltwoboundaryb}}^2\lesssim (1+{h})\|\ph\|^2_{\Ltwovector}+\|\uh-\uhath\|^2_{L^2(\partial \Th)}+\left(\int_{\Gamma} \uhath\,ds\right)^2
    \end{equation}
    and 
    \begin{equation}\label{eq:hybrid_trace_uhat_ph_intro}
        \|\uhath\|_{{\Ltwoboundaryb}}^2\lesssim (1+{h})\|\ph\|^2_{\Ltwovector}+\|\uh-\uhath\|^2_{L^2(\partial \Th)}+ \left(\int_{\Gamma} \uhath\,ds\right)^2.
    \end{equation}
These inequalities are crucial for establishing the stability of HDG-based numerical schemes. In practical applications, where $h$ is typically small, the factors $(1+h^2)$ and $(1+h)$ in \eqref{eq:poincare_mean2_ph_intro}–\eqref{eq:hybrid_trace_uhat_ph_intro} can be replaced by $1$. Throughout the paper, we retain the notation $h$ to clearly indicate the order of each term in the estimates.

{Towards the end of the introduction, we compare this work with two closely related studies. The first is \cite{jiang2023stability}, which also addresses the stability of the HDG method. The key distinction between that work and ours is that \cite{jiang2023stability} establishes the Poincaré inequality only under the assumption of a zero trace condition and relies on a translation argument to prove stability. In contrast, this work formulates a Poincaré inequality that depends solely on integrals over the domain or its boundary, making it more versatile and directly applicable in practical settings. Another closely related study is \cite{cockburn2018discrete}, in which the authors derive a Poincaré-type inequality by estimating the difference between a piecewise polynomial function and its local average over either the interior or boundary of a simplex. However, their work does not extend to estimating the corresponding average value, a gap that this study addresses. As a result, this work can be viewed as a direct continuation of \cite{cockburn2018discrete}, providing a more comprehensive functional tool that enhances applicability and facilitates further developments in the stability analysis of HDG methods.}

The rest of this paper is organized as follows: Section \ref{sec:pre} lays the groundwork for our study, covering essential notations like domain discretization and function spaces, mesh assumptions throughout our analysis, key technical lemmas aiding in the proof of our main results, and an introduction to the Crouzeix-Raviart space due to its critical role in our analysis. The Poincaré inequality for hybridizable spaces will be established in Section \ref{sec:poincare}, where an averaging technique and a lifting from piece-wise constant space to the Crouzeix-Raviart space will be developed as key preliminary steps for the proof. Section \ref{sec:trace} will extend the discussion to trace inequalities, employing a similar approach to the Poincaré inequality proof by utilizing a Crouzeix-Raviart element to bridge hybridizable and classical nonconforming spaces. Finally, Section \ref{sec:HDG} applies these findings to investigate the stability of the HDG method for second-order elliptic equations, 
obtaining a variant of these inequalities specifically designed for the HDG method.

\section{Preliminary}\label{sec:pre}

This section is dedicated to presenting the foundational preliminaries necessary for this paper. We will outline the notations and general assumptions frequently used throughout. Additionally, we will review several technical lemmas well-known in finite element analysis that will be applied in later discussions. A concise overview of the Crouzeix-Raviart space is also provided, due to its importance in our analysis.

\subsection{Notations and Assumptions}\label{subsec:notation}

\subsubsection{Space Discretization}

Consider a domain $\Omega$, which is an open connected polyhedral region in $\Rd$. Here, $d$ can either be 2 or 3. We define $\Th$ as a shape-regular triangulation of $\Omega$. (The exact definition of shape-regularity will be given in Section \ref{subsec:assumptions}) This means $\Th$ consists of triangles when $d=2$ and tetrahedrons for $d=3$. Each simplex in the triangulation is denoted as $K$, and so the entire domain can be written as $\Bar{\Omega} = \bigcup\limits_{K\in\Th} K$. When we have a face $e$ appearing as the intersection of two adjacent simplex, labeled as $K^+$ and $K^-$ ($e = \overline{K^+}\bigcap \overline{K^-}$), this face $e$ is called an interior face of the triangulation $\Th$. The set of all such interior faces is noted as $\pThi$. The faces that lie on the boundary are collected under $\pThb$. Hence, all faces within $\Th$ can be collectively described as $\partial \Th = \pThi \cup \pThb$. We call the collection of all the faces in the mesh to be skeleton. Additionally, the outward normal vectors for the simplexes $K^+$ and $K^-$ are represented by $\pmb{n}^+$ and $\pmb{n}^-$, respectively. However, in practice, we often drop the superscripts and simply use $\pmb{n}$ to denote the outward normal vector for a simplex $K$ at any given face. 

Regarding the mesh size of the triangulation $\Th$, as mentioned above, we use $h_K$ to indicate the diameter of a simplex $K$. This diameter is defined as the greatest distance between any two points within the simplex $K$.

\subsubsection{Function Spaces}\label{subsubsec:function_spaces}

The spaces that we will repeatedly utilize are listed here. Boldface notation will be used to indicate vector-valued functions or their corresponding spaces. The set of square integrable functions over $\Th$ within $\Omega$ is represented as $\Ltwo$. Similarly, square integrable functions on the face space, $\partial \Th$, are denoted by $\Ltwoboundary$. When addressing the function space associated with a specific element $K$, we modify the second component of this notation to reflect the domain of that element. For instance, $L^2(K)$, $L^2(\partial K)$, and $L^2(e)$ are used to represent the spaces of square integrable functions on $K$, on the boundary of $K$, and on a face $e$ belonging to $\partial \Th$, respectively. These spaces are defined with specific norms: for functions $\uh\in \Ltwo$ and $\uhath\in \Ltwoboundary$,

\begin{equation*}
    \|\uh\|_{\Ltwo} = \left(\sum_{K\in\Th}\|\uh\|^2_{L^2(K)}\right)^{\frac{1}{2}},\quad\quad \|\uhath\|_{\Ltwoboundary}=\left(\sum_{K\in\Th}\|\uhath\|^2_{L^2(\Kboundary)}\right)^{\frac{1}{2}}.
\end{equation*}
Moreover, when evaluating the $L^2$ norm of a function that is defined over the whole mesh that exists within the interior part of each element, specifically $\uh\in\Ltwo$, we can define its $L^2$ norm over the skeleton, $\partial\Th$, as well, by considering its trace on each face {(The trace is well-defined since, on each simplex $K$, $u_h$ is a polynomial.)}. Specifically, for $\uh\in\Ltwo$, its norm on $\partial\Th$ is expressed as $\|\uh\|_{\Ltwoboundary}$. {To ensure the definition is valid, for a given function $f \in \Uhk$, and for a face $e$ that serves as the boundary between two elements $K^+$ and $K^-$ with respective outward normal vectors $\pmb{n}^+$ and $\pmb{n}^-$ (i.e., $e = \overline{K^+} \cap \overline{K^-}$), we denote by $f^+$ and $f^-$ the values of $f$ on the sides of $K^+$ and $K^-$, respectively. Then the norm $\|\uh\|_{L^2(\partial \Th)}$ can be detailed as}:
\begin{equation*}
    \begin{aligned}
        \|\uh\|_{\Ltwoboundary}
        & = \left(\sum_{K\in\Th} \|\uh\|_{L^2({\partial K})}^2\right)^{\frac{1}{2}}\\
        &=\left(\sum_{e\in\pThi}\|{\uh^+} \pmb{n}_e^+\|^2_{L^2(e)}+\sum_{e\in\pThi}\|{\uh^-} \pmb{n}_e^-\|^2_{L^2(e)}+\sum_{e\in\pThb}\|\uh \pmb{n}_e\|^2_{L^2(e)}\right)^{\frac{1}{2}},
    \end{aligned}
\end{equation*}
where $\pmb{n}_e^+$ and $\pmb{n}_e^-$ denote the outward normal vectors on each side of an interior face $e$, and $\pmb{n}_e$ represents the outward normal vector at the boundary for a boundary face $e$. {Meanwhile, following the previous notation, we denote by $u_h^+$ and $u_h^-$ the values of $u_h$ on the two sides of the face $e$, analogous to $f^+$ and $f^-$.
}  Therefore, $\uh$ on each interior face is computed twice, reflecting the contributions from both sides of the face. $(\cdot,\cdot)_X$ will be used to denote inner product in $L^2(X)$ when $X$ is a collection of simplexes while we use $\langle\cdot,\cdot\rangle_X$ if $X$ is one or a collection of faces.

We designate $\Hone$ to represent for piecewise $H^1$ functions which is defined as
\begin{equation}
\label{eq:Hone_def}
    \Hone =\left\{f\in \Ltwo: f_K = f|_K\in H^1(K), \forall K\in \Th \right\}.
\end{equation}
We will use the operator $\gradienth$ to denote the broken gradient operator \cite{peraire2008compact}. In this context, $\gradienth f$ and $\gradienth\cdot \pmb{f}$ refer to functions that, when restricted to an element $K$, equal $\nabla f$ and $\nabla\cdot \pmb{f}$, respectively. The semi-norm for $\Hone$ is given by
\begin{equation}
    \label{eq:Hone_seminorm}
  \left  \lvert f\right\rvert_{\Hone}=\left(\sum_{K\in\Th}\|\nabla\uh\|^2_{L^2(K)}\right)^{\frac{1}{2}} = \|\gradienth \uh\|_{L^2(\Th)}.
\end{equation}
We want to clarify for readers that the notation $H^1(\Omega)$, which will be discussed in Section \ref{sec:trace}, refers to the standard Sobolev space. Regarding the $L^2$ space, we take $L^2(\Omega)=L^2(\Th)$ and $L^2(\partial\Omega)=\Ltwoboundaryb$ and we will not distinguish between these notations.

Next, we focus on defining the hybridizable spaces, which are central to this paper. The piecewise polynomial space $\Uhk$ within the domain is delineated as follows:
\begin{equation}\label{eq:Uhk_def}
   \Uhk=\{ f\in \Ltwo: f|_{K}\in \Pk(K), \forall K\in \Th  \},
\end{equation}
and the piecewise polynomial space $\Fhk$ over the faces is outlined as:
\begin{equation}\label{eq:Fhk_def}
    \Fhk=\{ \hat{f}\in \Ltwoboundary: \hat{f}|_e\in \Pk({e}), \forall e\in \partial\Th  \}.
\end{equation}
Here, $\Pk$ denotes the collection of polynomials with degree at most $k$. The combined space, $\Xhk$, is thus formulated as the Cartesian product of $\Uhk$ and $\Fhk$:
\begin{equation}\label{eq:Xhk_def}
    \Xhk = \Uhk \times \Fhk.
\end{equation}
Our analysis will primarily focus on elements within $\Xhk$, denoted by $(\uh,\uhath)\in \Xhk$. Additionally, we introduce the concept of a vector-valued piecewise polynomial function space, $\Vhk$, defined as:
\begin{equation}\label{eq:Vhk_def}
    \Vhk=\{ \pmb{f}\in \Ltwovector: \pmb{f}|_{K}\in \Pkvector(K), \forall K\in \Th  \}.
\end{equation}
This space becomes relevant when employing a lifting operator to transform a scalar function, defined on the mesh skeleton, into a vector function that is defined on the entire mesh. More information on this will be provided in Section \ref{subsec:boundary_lifting}. Additionally, it will be used for the discussion on HDG formulations in Section \ref{sec:HDG}.

\subsubsection{Other Notations}

To avoid proliferation of constants, we will use the notation $\lesssim$ in this paper. Specifically, when we say $f_1\lesssim f_2$, it implies the existence of a constant $C>0$, which is independent of both $f_1$ and $f_2$, ensuring $f_1\leq C f_2$.

Additionally, we follow the standard notations for jump at the boundaries of elements consisting of piecewise continuous functions \cite{brenner2008mathematical}. {For a given function $f \in \Uhk$, and for a face $e$, the jump of $f$ across the face is defined as:
\begin{equation}\label{eq:jump_def}
    [[f]]_e = f^+ \pmb{n}^+ + f^- \pmb{n}^-.
\end{equation}
Here, $e = \overline{K^+} \cap \overline{K^-}$ denotes the shared boundary between two elements $K^+$ and $K^-$ with respective outward normal vectors $\pmb{n}^+$ and $\pmb{n}^-$, and $f^+$ and $f^-$ represent the values of $f$ on the sides corresponding to $K^+$ and $K^-$, respectively, as defined earlier.}

We employ $|\cdot|$ to signify the magnitude or measure of an object. For instance, $|K|$ refers to the $d$-dimensional measure of $K$, while $|\partial K|$ and $|e|$ relate to the $(d-1)$-dimensional measures of $\partial K$ and a face $e$, respectively. Based on the shape regularity assumption, which will be detailed in Section \ref{subsec:assumptions}, we can assert that $|K|\propto (h_K)^d$ and $|\Kboundary|\propto (h_K)^{d-1}$.

For integral notation, we use $dx$ when referring to spatial integrals and $ds$ for integrals taken over faces. We also use $f|_K$ and $f|_e$ to denote the restriction of a function on a simplex $K$ or a face $e$, respectively.

\subsubsection{Mesh Assumptions}\label{subsec:assumptions}

In this section, we establish certain mesh assumptions. These criteria are broadly applicable across a range of meshes and are commonly used in analyses of DG and HDG methods, as seen in \cite{brenner2003poincare,cockburn2011analysis}.

\begin{enumerate}
    \item[\textbf{A1}] \textbf{(Shape Regularity Assumption):} There exists a constant $\kappa_{\mathcal{T}} > 0$ such that for {any mesh $\Th$}, the minimum ratio of the volume of any simplex $K$ within the mesh $\mathcal{T}_h$ to the power of its diameter (cubed for $d=3$ or squared for $d=2$) is always above $\kappa_{\mathcal{T}}$:
\begin{equation}
    \min_{K \in \mathcal{T}_h} \frac{|K|}{\text{diam}(K)^d} \geq \kappa_{\mathcal{T}}.
\end{equation}

    \item[\textbf{A2}] { \textbf{(Quasi-Uniformity Assumption):}
There exists a constant $A_{\mathrm{qu}} > 0$ such that for any mesh $\mathcal{T}_h$ and for any elements $K, K' \in \mathcal{T}_h$, the following holds:
\[
  \frac{\mathrm{diam}(K)}{\mathrm{diam}(K')} \;\le\; A_{\mathrm{qu}}.
\]
Equivalently,
\[
  \frac{\min_{K \in \mathcal{T}_h}\!\mathrm{diam}(K)}{\max_{K \in \mathcal{T}_h}\!\mathrm{diam}(K)} 
  \;\ge\;\frac{1}{A_{\mathrm{qu}}}.
\]}

    \item[\textbf{A3}] \textbf{(Hanging Node Assumption):} The mesh does not contain hanging nodes.
\end{enumerate}

Two remarks are given here regarding these assumptions:

\begin{rem}
    The Shape Regularity Assumption (A1) equivalently introduces a constant $\theta_{\mathcal{T}} > 0$, which bounds the maximum diameter of any simplex $K$ relative to the diameter of the largest inscribed sphere in $K$, for all $h_K > 0$:
\begin{equation*}
    \max_{K \in \mathcal{T}_h} \frac{h_K}{\rho_K} \leq \theta_{\mathcal{T}},
\end{equation*}
where $\rho_K$ represents the inscribed sphere's diameter. It also equivalently establishes a constant $\varphi_{\mathcal{T}} > 0$, ensuring the minimum angle within any simplex $K$ remains above $\varphi_{\mathcal{T}}$. This angle is measured in radians for $d = 2$ or steradians for $d = 3$.
\end{rem}

\begin{rem}
   { The Quasi-Uniformity Assumption (A2) ensures that the sizes of the elements in the mesh are of the same order of magnitude, meaning no element has a diameter that is drastically smaller or larger than another. Throughout this paper, we typically use $h = \max_K h_K$, and (A2) guarantees that $h_K \lesssim h$ and $\frac{1}{h_K} \lesssim \frac{1}{h}$. We will use this directly without explicitly referencing it in later discussions.}

\end{rem}

\subsection{Technical Lemmas}
This section outlines several well-known results that pave the way for the proofs developed later in this paper. We begin with the discrete {local} trace theorem in triangular simplex, summarized in the lemma below.

\begin{lem}\label{lemma:trace_inequality_simplex}
Consider a simplex $K$ in $\mathbb{R}^{d}$, with $e$ representing one of its faces. For any function $f$ belonging to $\Pk(K)$, the following inequality holds true:
\begin{equation}
  \| f\|_{L^2(e)} \leq \left(\frac{(k + 1)(k + d)}{d}\right)^{\frac{1}{2}} \left(\frac{|e|}{|K|}\right)^{\frac{1}{2}} \| f \|_{L^2(
K)}   .
\end{equation}

\end{lem}

\begin{proof}
    For the proof, we refer readers to \cite{warburton2003constants}[Theorem 5].
\end{proof}

\begin{rem}
    \label{rem:trace_inequality_vector}
    The lemma mentioned primarily focuses on scalar-valued functions. To extend this principle to vector-valued functions, one can evaluate the inequality for each vector component separately and then combine the results. This approach leads to a vector-valued version of the inequality.
\end{rem}

\begin{rem}
    \label{rem:trace_inequality_hK}
    Considering a simplex $K$ and its face $e$, there is a proportional relationship between their measures, expressed as $|K| = c_dh_e|e|$. Here, $h_e$ represents the height from face $e$ within $K$, and $c_d$ is only dependent on the dimension $d$. With $h_K$ denoting the diameter of the simplex $K$, the {local} discrete trace inequality can be rephrased to reflect this geometrical relation under the shape regularity assumption, as follows:
   \begin{equation}
       \| f \|_{L^2(e)} \lesssim \frac{1}{h_K^{\frac{1}{2}}} \| f \|_{L^2(K)} .
   \end{equation}

\end{rem}

Next, we present the Poincaré inequality in a simplex $K$, which includes an estimate of the order of the Poincaré constant in this case.

\begin{lem}\label{lem:Poincare_Frederich_discrete_boundary_average}
Let $K$ be a simplex, $e$ is a face in $\Kboundary$, and $f \in H^1(K)$. We set
\begin{equation*}
f_{e}:= \frac{1}{|e|} \int_{e} f \, ds,\quad\quad f_K=\frac{1}{|K|}\int_K f\,dx.
\end{equation*}
They denote the average of $f$ over one face $e$ and the interior of simplex $K$, respectively. Then the following estimates hold
\begin{equation}\label{eq:mean_difference_optimal_Poincare_constant}
    \int_K \left(f_K-f_e  \right)^2\,dx \lesssim \, (h_K)^2 \int_K |\nabla f|^2 \, dx,
\end{equation}
\begin{equation}\label{eq:mean_optimal_Poincare_constant}
\int_K [f - f_{K}]^2 \, dx \lesssim \, (h_K)^2 \int_K |\nabla f|^2 \, dx.
\end{equation}
and
\begin{equation}\label{eq:mean_boundary_optimal_Poincare_constant}
\int_K [f - f_{e}]^2 \, dx \lesssim \, (h_K)^2 \int_K |\nabla f|^2 \, dx.
\end{equation}

\end{lem}
\begin{proof}
    See \cite{vohralik2005discrete}[Lemma 4.1] for \eqref{eq:mean_difference_optimal_Poincare_constant} and \eqref{eq:mean_boundary_optimal_Poincare_constant} and then \eqref{eq:mean_optimal_Poincare_constant} will follow as a consequence. Or see \cite{payne1960optimal} and \cite{Esposito2013} for \eqref{eq:mean_optimal_Poincare_constant} directly.
\end{proof}

This leads us to understand that the $L^2$ norm of the difference between a function and its mean value (either averaged over the entire simplex or just on one face of it) can be bounded by the function's $H^1$ semi-norm, multiplied by a constant that depends on the simplex's diameter.

\subsection{Crouzeix–Raviart Space}

Introduced by Crouzeix and Raviart in the early 1970s \cite{crouzeix1973conforming}, the Crouziex-Raviart (CR) finite element space is an important development in the area of non-conforming $P_1$ finite elements. Characterized by its application to both triangular ($d=2$) and tetrahedron ($d=3$) cases, the CR space uniquely defines its degrees of freedom through the evaluation of functions at the midpoint of edges or faces. This distinctive approach results in element functions that maintain continuity exclusively at these midpoints, different from the traditional conforming elements \cite{brenner2008mathematical,ciarlet2002finite,zienkiewicz2005finite} which are continuous across the entire element.

As a consequence, the discontinuity outside the midpoints of edges or faces makes CR finite element function not an element of the Sobolev space $H^1(\Omega)$ which is the standard space for second-order elliptic equations to be posed in \cite{evans2022partial,gilbarg1977elliptic}. This difference underscores the non-conforming nature of the CR space. The theoretical analysis and evolution of CR space are extensively discussed in literature. We refer readers to \cite{apel2001crouzeix,brenner2015forty,di2015extension,hansbo2003discontinuous} and the references therein.

Here, we give the precise definition of the CR space:

\begin{equation}
\label{eq:CR_def}
    \CRspace=\{ f\in \Ltwo: f|_K \in \Pone(K), \int_e [[f]]_e\,ds=0 \text{ for all } e\in\pThi\}.
\end{equation}
The condition that the integral of the function's jump across any interior edge is zero underlines that the CR space's degrees of freedom are centered on the edges' midpoints.

In this research, the CR space is utilized for the interpolation of a function $\hat{\mu}$, defined on the mesh skeleton, into a $P_1$ element within the CR space. The specifics of this interpolation method will be detailed in Section \ref{subsec:CRLifting}.

\section{Discrete Poincar\'e Inequality}\label{sec:poincare}

This section focuses on establishing one of our main results, the Poincaré inequality, within the hybridizable space $\Xhk$. Our approach involves linking hybridizable spaces to DG spaces via a lifting operator. This connection is set up by utilizing the Crouzeix-Raviart element as an intermediary. 

\subsection{Discrete Poincaré Inequality for Piecewise $H^1$ Functions}\label{subsec:discrete_poincare_brenner}

To start, we revisit the discrete Poincaré-Friedrichs inequalities that applied to classical nonconforming finite element methods and discontinuous Galerkin methods, as introduced in \cite{brenner2003poincare}. These inequalities will serve as crucial tools in deriving the Poincaré-Friedrichs inequalities for hybridizable spaces.

\begin{lem}\label{lem:brenner_poincare}
The following are the Poincar\'e--Friedrichs inequalities for $f\in \Hone$ where $\Hone$ is the space of piecewise $H^1$ functions defined in \eqref{eq:Hone_def}:

\begin{equation}\label{eq:discrete_poincare}
    \|f\|_{\Ltwo}^2 \lesssim \left[ |f|_{H^1({\Th})}^2 + \sum_{e \in \pThi} |e|^{d/(1-d)} \left| \int_{e} [[f]]_e \, ds \right|^2 + \left( \int_{\Omega} f \, dx \right)^2 \right],
\end{equation}

\begin{equation}\label{eq:discrete_friedrichs}
    \|f\|_{\Ltwo}^2 \lesssim \left[ |f|_{H^1({\Th})}^2 + \sum_{e \in \pThi} |e|^{d/(1-d)} \left| \int_{e} [[f]]_e \, ds \right|^2 + \left( \int_{\Gamma} f \, ds \right)^2 \right],
\end{equation}
where $|e|$ represents the $(d-1)$-dimensional measure of the face $e$, and $[[f]]_e$ signifies the jump of the function $f$ across the face $e$, as defined in \eqref{eq:jump_def}. The positive constant, not explicitly mentioned due to the use of the symbol $\lesssim$, relies solely on the shape regularity of the mesh ${\Th}$. {And $\Gamma$ is a subset of $\partial\Omega$ that has a positive measure}.

\end{lem}

\begin{proof}
    See \cite{brenner2003poincare}.
\end{proof}
The following corollary is an immediate result when the integral of jump at each interior face vanishes, namely,
\begin{equation}\label{eq:zero_jump_condition}
    \int_e [[f]]_e\,ds={0}
\end{equation}
for every $e\in\pThi$.

\begin{cor}\label{cor:Poincare_no_jump}
    If condition \eqref{eq:zero_jump_condition} holds, then \eqref{eq:discrete_poincare}-\eqref{eq:discrete_friedrichs} will reduce to 
    \begin{equation}\label{eq:discrete_poincare_no_jump}
    \|f\|_{\Ltwo}^2 \lesssim |f|_{H^1({\Th})}^2  + \left( \int_{\Omega} f \, dx \right)^2,
\end{equation}
and
\begin{equation}\label{eq:discrete_friedrichs_no_jump}
    \|f\|_{\Ltwo}^2 \lesssim  |f|_{H^1({\Th})}^2 + \left( \int_{\Gamma} f \, ds \right)^2.
\end{equation}
\end{cor}

For any function $f$ within the CR space, as defined in \eqref{eq:CR_def}, it satisfies the condition given in \eqref{eq:zero_jump_condition}. Therefore, the two types of Poincaré-Friedrichs inequalities mentioned above, \eqref{eq:discrete_poincare_no_jump} and \eqref{eq:discrete_friedrichs_no_jump}, which apply to cases without jumps in integral sense, are naturally applicable to a CR element.

We want to point out that Lemma \ref{lem:brenner_poincare} is broadly applicable and serves as a cornerstone in the theory of the DG method \cite{di2011mathematical,ern2021finite,john2016finite,riviere2008discontinuous}. This lemma allows us to generalize the concept of derivatives to discontinuous spaces and lays down a framework for designing a generalized gradient operator. It illustrates what gradients look like within such spaces. However, directly applying these principles to hybridizable spaces does not suffice to achieve the inequalities presented in \eqref{eq:poincare_mean1_intro}-\eqref{eq:poincare_boundary_mean_intro}. In fact, when considering an element $(\uh,\uhath)\in\Xhk$, if the jump term $ \sum_{e \in \pThi} |e|^{d/(1-d)} \left| \int_{e} [[\uh]]_e \, ds \right|^2$ in \eqref{eq:discrete_poincare} and \eqref{eq:discrete_friedrichs} remains, the best estimate we could expect for this term would be
\begin{equation*}
    \sum_{e \in \pThi} |e|^{d/(1-d)} \left| \int_{e} [[\uh]]_e \, ds \right|^2\leq \sum_{e \in \pThi} |e|^{1/(1-d)} \left| \int_{e} \left([[\uh]]_e \right)^2\, ds \right|\lesssim  \sum_{e \in \pThi} {h_K}^{\frac{1}{1-d}}\|\uh-\uhath\|^2_{L^2(e)}.
\end{equation*}
For example, in the two-dimensional case ($d=2$), bounding the associated jump term requires an $O(\frac{1}{h})$ factor, which is typically difficult to achieve in an HDG formulation. This challenge indicates that the jump term $\sum_{e \in \pThi} |e|^{d/(1-d)} \left| \int_{e} [[\uh]]_e \, ds \right|^2$
should be treated separately. Accordingly, we introduce the CR lifting operator in the following subsection to effectively handle this term.

\subsection{CR Lifting Operator} 
\label{subsec:CRLifting}

Here we introduce a lifting operator mapping a piece-wise constant function defined on skeleton of the mesh to be a function defined in the CR space. We define $\CRlifting: \Pzero(\partial \Th)\to \CRspace$ as 
\begin{equation}\label{eq:CR_lifting_def}
    \CRlifting(\hat{\mu})(c_e) = \hat{\mu}|_{e}
\end{equation}
for every $\hat{\mu} \in \Pzero(\partial\Th)$ and face $e \in \partial\Th$,  
{where $\Pzero(\partial \Th)$ denotes the space of piecewise constant functions on the skeleton of $\Th$}. Here $c_e$ denotes the center of the face $e$ and the left-hand side of \eqref{eq:CR_lifting_def} is to evaluate $ \CRlifting(\hat{\mu})$ at point $c_e$. In other words, the lifting operator defines a CR element by determining its degree of freedom lying on the centers of each side of each element.

The following result describes the quantity relation between $\hat{\mu}$ and $ \CRlifting(\hat{\mu})$ when it is restricted in $K$.

\begin{lem}
    \label{lem:CRLifting_estimate}
    We take the restriction of $\hat{\mu}\in\Pzero(\partial\Th)$ on  an element $K$ as $\hat{\mu}_K=\hat{\mu}|_K$, then
    \begin{equation}
        \|\hat{\mu}_K\|^2_{L^2(
    \Kboundary)}\leq \|\CRlifting(\hat{\mu}_K)\|^2_{L^2(\partial K)} = \|\CRlifting(\hat{\mu})|_K\|^2_{L^2(\partial K)}.
    \end{equation}
\end{lem}

\begin{proof}
    Consider a face $e$ within $\Kboundary$ and let $\hat{\mu}_{K,e}$ represent the value of $\mu$ on face $e$. We can express this term in another way as
    \begin{equation*}
  \hat{\mu}_{K,e}=\frac{1}{|e|}\int_e \CRlifting(\hat{\mu}_K)\,ds
\end{equation*}
due to the linearity of $\CRlifting(\hat{\mu}_K)$ and the definition of the CR lifting operator. Applying the Cauchy-Schwarz inequality, we get
\begin{equation}
    \label{eq:uhat_average_lifting_comparison_estimate}(\hat{\mu}_{K,e})^2=\frac{1}{|e|^2}\left[\int_e \CRlifting(\hat{\mu}_K)\,ds\right]^2 \leq \frac{1}{|e|}\int_e \left[\CRlifting(\hat{\mu}_K)\right]^2\,ds=\frac{1}{|e|}\left\|\CRlifting(\hat{\mu}_K)\right\|^2_{L^2(e)}.
\end{equation}
Following this,
\begin{align*}
    \|\hat{\mu}_K\|^2_{L^2(\Kboundary)} = \sum_{e\in \Kboundary}|e|(\hat{\mu}_{K,e})^2\leq\sum_{e\in \Kboundary}\left\|\CRlifting(\hat{\mu}_K)\right\|^2_{L^2(e)}=\|\CRlifting(\hat{\mu}_K)\|^2_{L^2(\partial K)}.
\end{align*}
This concludes the proof.
\end{proof}

\subsection{Poincar\'e Inequalities in $\Xhk$}

In this part, our goal is to prove the Poincaré inequalities for the space $\Xhk$. Before proceeding with the proof, it's crucial to recognize that $\uhath\in\Fhk$ is generally a {piecewise polynomial function instead of a piecewise constant function}. For the CR lifting operator to be applicable, we need to transform these piecewise polynomial functions into piecewise constant ones. Thus, for a boundary face $e$ and a function $\uhath\in\Fhk$, we take:
\begin{equation}
\label{eq:side_average_hat}
    \uhathaverageEdge = \frac{1}{|e|}\int_e \uhath\,ds.
\end{equation}
Then $\uhataverage\in \Pzero(\partial \Th)$ is introduced as a piecewise constant function that averages $\uhath$ on each boundary segment, defined by:
    \begin{equation}\label{eq:uhat_average_def}
    \uhataverage|_e = \uhathaverageEdge.
\end{equation}
This procedure converts $\uhath$ into the piecewise constant function $\uhataverage$ within $\Fhk$, facilitating the unique definition of a corresponding CR element by the CR lifting operator $\CRlifting$.

With these preparations, we present the following findings, detailing the connection between the $L^2$ norm of $\uh$ and that of the function derived from $\CRlifting(\uhataverage)$, along with the jump terms existing on the faces.

\begin{prop}
\label{prop:L2_estimate}
    Let $(\uh,\uhath)\in \Xhk$. Then the following local inequality holds in each element $K$:
     \begin{equation}\label{eq:L2_estimate_local_K}
        \|\uh\|^2_{\LtwoK}\lesssim h^2_K\lvert\uh \rvert_{H^1(K)}^2 +h_K \| \uh-\uhath\|^2_{L^2(\Kboundary)}+ \left\|\CRlifting(\uhataverage)\right\|_{\LtwoK}^2,
    \end{equation}
    and the global inequality will naturally hold as well:
    \begin{equation}
        \label{eq:L2_estimate_global_Th}\|\uh\|^2_{\Ltwo}\lesssim h^2_K\lvert\uh \rvert_{\Hone}^2 +h_K \| \uh-\uhath\|^2_{L^2(\partial\Th)}+ \left\|\CRlifting(\uhataverage)\right\|_{\Ltwo}^2.
    \end{equation}
\end{prop}

\begin{proof}
    We start by restricting our scope in an element $K$. In each $K$, we can split $\|\uh\|_{L^2(K)}$ into two parts as
    \begin{equation}\label{eq:uh_split}
        \|\uh\|_{L^2(K)}\lesssim\|\uh-\uhaverageBoundary\|_{L^2(K)}+\|\uhaverageBoundary\|_{L^2(K)}.
    \end{equation}
    where 
 \begin{equation}\label{eq:uhaverageBoundary_def}
\uhaverageBoundary:= \frac{1}{|\partial K|} \int_{\partial K} \uh \, ds
\end{equation}
is defined as the average of $\uh$ over $\partial K$. Using the Poincar{\'e} inequality in a single simplex (Lemma \ref{lem:Poincare_Frederich_discrete_boundary_average}), the first quantity on the right-hand side of \eqref{eq:uh_split} can be controlled as:
    \begin{align}
        \label{eq:difference_u_uaverage_estimate}
       \|\uh-\uhaverageBoundary\|_{L^2(K)}&\lesssim h_K\|\gradienth \uh\|_{L^2(K)}.
    \end{align}
    For the second quantity, since $u_{\Kboundary}$ is a constant, we have
    \begin{equation*}
        \|\uhaverageBoundary\|_{L^2(K)}^2=\lvert K \rvert \,\uhaverageBoundary^2.
    \end{equation*}
This leads us to the objective of estimating $\uhaverageBoundary$ to effectively bound $\|\uh\|_{L^2(K)}$. By revisiting its definition, we can reinterpret this as the cumulative sum of integrals of $\uh$ across the different faces $e_i$ of element $K$, yielding the subsequent formulation:

    \begin{align*}
        \uhaverageBoundary=\frac{1}{|\Kboundary|} \int_{\Kboundary} \uh \, ds&=\frac{1}{|\Kboundary|} \sum_{i=1}^{d+1}\int_{e_i} \uh \, ds\\
        &=\frac{1}{|\Kboundary|} \sum_{i=1}^{d+1}\int_{e_i} (\uh -\uhathaverageEdgei)\,ds +\frac{1}{|\Kboundary|} \sum_{i=1}^{d+1}\int_{e_i}\uhathaverageEdgei\, ds
    \end{align*}
where definition of  $\uhathaverageEdgei$ follows from \eqref{eq:side_average_hat}.   As $\int_{e_i}\uhathaverageEdgei\,ds=\int_{e_i}\uhath\,ds$, $ \uhaverageBoundary$ can be written as:
    \begin{align*}
         \uhaverageBoundary & =\frac{1}{|\Kboundary|} \int_{\Kboundary} (\uh-\uhath)\,ds  +\frac{1}{|\Kboundary|} \sum_{i=1}^{d+1}\int_{e_i}\uhathaverageEdgei\, ds.
    \end{align*}
Applying the Cauchy-Schwarz inequality along with the assumption of shape regularity enables us to perform the following calculation:

    \begin{equation}\label{eq:uhaverageBoundary_estimate_point}
    \begin{aligned}
        (\uhaverageBoundary)^2&\lesssim \frac{1}{|\partial K|^2}\left[\int_{\Kboundary} (\uh-\uhath)\,dx \right]^2+\frac{1}{|\partial K|^2}\left[ \sum_{i=1}^{d+1}\int_{e_i}\uhathaverageEdgei\, dx \right]^2\\
        &\leq \frac{1}{|\partial K|}\int_{\Kboundary} (\uh-\uhath)^2\,dx + \frac{1}{|\partial K|^2}\left( \sum_{i=1}^{d+1}|e_i|\,\uhathaverageEdgei \right)^2\\
        &\lesssim \frac{1}{|\partial K|}\| \uh-\uhath\|^2_{L^2(\Kboundary)}+ \frac{1}{|\partial K|^2} \left(\sum_{i=1}^{d+1}|e_i|^2\right)\,\left(\sum_{i=1}^{d+1}\,\uhathaverageEdgei^2\right)\\
        &\lesssim \frac{1}{|\partial K|}\| \uh-\uhath\|^2_{L^2(\Kboundary)}+\sum_{i=1}^{d+1} \,(\uhathaverageEdgei)^2.
    \end{aligned}
    \end{equation}
Hence, using the principle of shape regularity once more, we obtain:
    \begin{equation}
        \label{eq:u_boundary_average_estimate}
        \|\uhaverageBoundary\|^2_{L^2(K)}=\lvert K \rvert\, \uhaverageBoundary^2\lesssim h_K \| \uh-\uhath\|^2_{L^2(\Kboundary)}+ h_K\|\Bar{\hat{u}}_{h}\|^2_{L^2(\partial K)},
    \end{equation}
where $\uhataverage$ is defined in \eqref{eq:uhat_average_def}. Additionally, we have used the following fact to deduce \eqref{eq:u_boundary_average_estimate}
\begin{equation*}\label{eq:uhataverage_definition}
   \|\Bar{\hat{u}}_{h}\|^2_{L^2(\partial K)} = \sum_{i=1}^{d+1} \left|e_i\right|\,\uhathaverageEdgei^2
\end{equation*}
since $\uhathaverageEdgei$ is a constant for each $i$.

Then we use the lifting operator $\CRlifting$ to define a function $\CRlifting(\uhataverage)$ in $\CRspace$. According to Lemma \ref{lem:CRLifting_estimate}, we get that in each element $K$,
\begin{equation}\label{eq:uhat_lifting_estimate}
    \|\uhataverage\|^2_{L^2(\Kboundary)}\leq  \left\|\CRlifting(\uhataverage)\right\|_{L^2(\Kboundary)}^2.
\end{equation}
Now combining \eqref{eq:uhat_lifting_estimate} with Lemma \ref{lemma:trace_inequality_simplex} (discrete trace inequality in simplex), we can rewrite \eqref{eq:u_boundary_average_estimate} as
\begin{equation}\label{eq:uhaverageboundary_estimate}
  \begin{aligned}
    \|\uhaverageBoundary\|_{L^2(K)}^2&\lesssim h_K \| \uh-\uhath\|^2_{L^2(\Kboundary)}+ h_K\|\uhataverage\|^2_{L^2(\partial K)}\\
    &\leq h_K \| \uh-\uhath\|^2_{L^2(\Kboundary)}+ h_K\left\|\CRlifting(\uhataverage)\right\|_{L^2(\Kboundary)}^2\\
    &\lesssim h_K \| \uh-\uhath\|^2_{L^2(\Kboundary)}+ \left\|\CRlifting(\uhataverage)\right\|_{L^2(K)}^2.
\end{aligned}  
\end{equation}
By inserting the estimates from \eqref{eq:uhaverageboundary_estimate} and \eqref{eq:difference_u_uaverage_estimate} into \eqref{eq:uh_split}, we get \eqref{eq:L2_estimate_local_K}. And then summing the results over all elements $K$ in $\Th$, \eqref{eq:L2_estimate_global_Th} follows.

\end{proof}

In this analysis, we encounter the term $\left\|\CRlifting(\uhataverage)\right\|_{L^2(K)}^2$. This is where the Poincaré inequality for nonconforming spaces with no jump term, as outlined in Corollary \ref{cor:Poincare_no_jump}, becomes relevant because this is a CR element. Consequently, an integral involving $\CRlifting(\uhataverage)$ could emerge. To express this term in a form more consistent with the classical Poincaré inequality, we examine the difference between $\int_\Omega\CRlifting(\uhataverage)\,dx $ and $\int_\Omega\uh\,dx$, leading to the subsequent finding:

\begin{lem}
    \label{lem:average_difference_u_uhat}
    The difference between integral of $\CRlifting(\uhataverage)$ and $\uh$ can be controlled as
    \begin{equation}
         \left(\int_\Omega\CRlifting(\uhataverage)\,dx \right)^2\lesssim ({h})^2\lvert\uh \rvert_{\Hone}^2 +{h} \| \uh-\uhath\|^2_{L^2(\partial\Th)}+\left( \int_\Omega\uh\,dx\right)^2.
    \end{equation}
    
\end{lem}
\begin{proof}
For simplicity of notation, we use $\omega_h$ to denote the piecewise linear function $ \CRlifting(\uhataverage)$ in the sense that $\omega_h|_K = \CRlifting(\uhataverage)|_K$.
    Then 
    \begin{align*}
        \int_\Omega \CRlifting(\uhataverage)\,dx-\int_\Omega \uh\,dx&=\int_\Omega\omega_h\,dx-\int_\Omega \uh\,dx\\
        &=\sum_{K\in\Th} \int_K \left(\omega_h-\uh\right)\,dx = \sum_{K\in\Th} \int_K\left(\omegahaverageK-\uhaverageK\right)\,dx
    \end{align*}
where $\omegahaverageK$ and $\uhaverageK$ denotes their averages over the simplex $K$. Since $\omega_h$ is a linear function, $\omegahaverageK$ can be evaluated via its value on the vertices and so we can compute that
\begin{equation}
    \omegahaverageK = \frac{1}{d+1}\sum_{i=1}^{d+1} \uhathaverageEdgei.
\end{equation}
With this, we have
\begin{align*}
    \int_\Omega \CRlifting(\uhataverage)\,dx-\int_\Omega \uh\,dx&=\frac{1}{d+1}\sum_{K\in\Th} \sum_{i=1}^{d+1}\int_K\left(\uhathaverageEdgei-\uhaverageK\right)\,dx\\
    &=\frac{1}{d+1}\sum_{K\in\Th} \sum_{i=1}^{d+1}\left[\int_K\left(\uhathaverageEdgei-\uhaverageEdgei\right)\,dx+\int_K\left(\uhaverageEdgei-\uhaverageK\right)\,dx\right]\\
    &={\frac{1}{d+1}\sum_{K\in\Th} \sum_{i=1}^{d+1} \frac{|K|}{|e_i|}\int_{e_i} \left(\uhath-\uh\right)\,ds +\frac{1}{d+1}\sum_{K\in\Th} \sum_{i=1}^{d+1}\int_K\left(\uhaverageEdgei-\uhaverageK\right)\,dx}\\
    &:=I_1 + I_2
\end{align*}
where $\uhaverageEdgei$, following the definition introduced in Lemma \ref{lem:Poincare_Frederich_discrete_boundary_average}, is defined as average of $\uh|_K$ on the side $e_i$ with the simplex $K$,
\begin{equation*}
    \uhaverageEdgei = \frac{1}{|e_i|}\int_{ e_i } \uh|_K\,ds.
\end{equation*}
Therefore, we can control $I_1$, using the shape-regularity assumption and Cauchy-Schwarz inequality, as
\begin{align*}
(I_1)^2={\frac{1}{(d+1)^2}\left[\sum_{K\in\Th} \sum_{i=1}^{d+1} \frac{|K|}{|e_i|}\int_{e_i} \left(\uhath-\uh\right)\,ds\right]^2}
    &\lesssim { \left[\sum_{K\in\Th} \frac{|K|}{|\Kboundary|}\int_{\partial_K} \left|\uhath-\uh\right|\,ds\right]^2}\\
    &\lesssim {\left(\sum_{K\in\Th}|K|\right)\,\left[\sum_{K\in \Th} \frac{|K|}{|\partial K|^2}\left(\int_{\partial K} |\uhath-\uh|\,ds\right)^2 \right]}  \\
    &\lesssim {\sum_{K\in\Th} \frac{|K|}{|\Kboundary|}\int_{\partial_K} \left|\uhath-\uh\right|^2\,ds }\\
    &\lesssim {h} \| \uh-\uhath\|^2_{L^2(\partial\Th)}
\end{align*}
{where $h=\max_K h_K$. The last inequality holds as Assumption A1 implies $\frac{|K|}{|\partial K|}$ is bounded above.} On the other hand, $I_2$ can be controlled using \eqref{eq:mean_difference_optimal_Poincare_constant} from Lemma \ref{lem:Poincare_Frederich_discrete_boundary_average} as 
\begin{align*}
    (I_2)^2=\frac{1}{(d+1)^2}\left[\sum_{K\in\Th} \sum_{i=1}^{d+1}\int_K\left(\uhaverageEdgei-\uhaverageK\right)\,dx\right]^2& 
 \lesssim {\left[\sum_{K\in\Th}\int_K \left(\uhaverageEdgei-\uhaverageK\right)\,dx\right]^2}\\
 &\lesssim {\left(\sum_{K\in\Th} |K|\right)\,\left[\sum_{K\in\Th}\frac{1}{|K|}\left(\int_K\left|\uhaverageEdgei-\uhaverageK\right|\,dx \right)^2\right]}\\
    &\lesssim \sum_{K\in\Th} \int_K \left|\uhaverageEdgei-\uhaverageK\right|^2\,dx\\
    &\lesssim  {\sum_{K\in\Th}(h_K)^2\|\gradienth\uh\|_{L^2(K)}^2}\\&{\lesssim  h^2\lvert  \uh\rvert_{\Hone}^2}.
\end{align*}
Combining the estimate for $I_1$ and $I_2$, we get
\begin{align*}
    \left(\int_\Omega\CRlifting(\uhataverage)\,dx \right)^2&\lesssim \left(\int_\Omega \CRlifting(\uhataverage)\,dx-\int_\Omega \uh\,dx\right)^2 + \left( \int_\Omega\uh\,dx\right)^2\\
    &\lesssim {h}^2\lvert\uh \rvert_{\Hone}^2 +{h} \| \uh-\uhath\|^2_{L^2(\partial\Th)}+\left( \int_\Omega\uh\,dx\right)^2.
\end{align*}
Here we finish the proof.
    
\end{proof}

Now we state the main result of this section.

\begin{thm}\label{thm:Poincare}
    Let $(\uh,\uhath)\in\Xhk$. Then the following Poincar\'e inequalities hold:
    \begin{equation}\label{eq:poincare_mean1}
          \|\uh\|^2_{\Ltwo}\lesssim {h}^2\lvert\uh \rvert_{\Hone}^2 +{h} \| \uh-\uhath\|^2_{L^2(\partial\Th)}+ \left| \CRlifting(\uhataverage) \right|_{\Hone}^2 + \left(\int_\Omega\CRlifting(\uhataverage)\,dx \right)^2
    \end{equation}
    and
    \begin{equation}\label{eq:poincare_boundary_mean}
          \|{\uh}\|^2_{\Ltwo}\lesssim {h}^2\lvert\uh \rvert_{\Hone}^2 +{h} \| \uh-\uhath\|^2_{L^2(\partial\Th)}+ \left| \CRlifting(\uhataverage) \right|_{\Hone}^2 + \left(\int_\Gamma\uhath\,ds \right)^2
    \end{equation}
    {where $\Gamma$ is combination of boundary faces that has a positive measure}, namely, $\Gamma=\bigcup_{i=1}^N e_i$ such that $e_i\in\pThb$ and $\{e_1,e_2,\cdots,e_N\}$ are different faces. In addition, the following variant of \eqref{eq:poincare_mean1} expressing in term of integral of $\uh$ also holds:
    \begin{equation}\label{eq:poincare_mean2}
          \|\uh\|^2_{\Ltwo}\lesssim {h}^2\lvert\uh \rvert_{\Hone}^2 +{h} \| \uh-\uhath\|^2_{L^2(\partial\Th)}+ \left| \CRlifting(\uhataverage) \right|_{\Hone}^2 + \left(\int_\Omega\uh\,dx \right)^2.
    \end{equation}
\end{thm}
\begin{proof}
    As $\CRlifting(\uhataverage)\in \CRspace$, estimate \eqref{eq:discrete_poincare_no_jump} and \eqref{eq:discrete_friedrichs_no_jump} in Corollary \ref{cor:Poincare_no_jump} hold for $\CRlifting(\uhataverage)$. Combining \eqref{eq:discrete_poincare_no_jump} with Proposition \ref{prop:L2_estimate}, \eqref{eq:poincare_mean1} will be immediately obtained. Similarly, combining \eqref{eq:discrete_friedrichs_no_jump} with Proposition \ref{prop:L2_estimate} will lead to
    \begin{equation*}
        \|\uh\|^2_{\Ltwo}\lesssim {h}^2\lvert\uh \rvert_{\Hone}^2 +{h} \| \uh-\uhath\|^2_{L^2(\partial\Th)}+ \left| \CRlifting(\uhataverage) \right|_{\Hone}^2 + \left(\int_\Gamma\CRlifting(\uhataverage)\,ds \right)^2.
    \end{equation*}
    Since $\Gamma$ is a combination of boundary faces, we have
    \begin{align*}
        \int_\Gamma\CRlifting(\uhataverage)\,ds & =\int_{\bigcup_{i=1}^N e_i}\CRlifting(\uhataverage)\,ds\\
        &=\sum_{i=1}^N \int_{e_i} \CRlifting(\uhataverage)\,ds=\sum_{i=1}^N \int_{e_i}\uhataverage\,ds=\sum_{i=1}^N \int_{e_i}\uhath\,ds=\int_\Gamma\uhath\,ds 
    \end{align*}
    and so \eqref{eq:poincare_boundary_mean} is obtained. 

    To the end, \eqref{eq:poincare_mean1} together with Lemma \ref{lem:average_difference_u_uhat} will immediately lead to \eqref{eq:poincare_mean2}.
\end{proof}

\section{Discrete Trace Inequality}\label{sec:trace}

In this section, we present an analogue of the trace theorem specifically designed for hybridizable spaces, which is an essential tool for analyzing boundary value problems. We want to highlight the difference between the notations $H^1(\Omega)$ and $L^2(\Omega)$, which indicate the standard Sobolev space and the square integrable space, respectively, and the specialized notations $\Hone$ and $\Ltwo$. Our discussion commences with a finding from \cite{brenner2015piecewise} that provides insight into $\Hone$, laying the groundwork for understanding this space through its relationship with functions belonging to $H^1(\Omega)$.

\begin{lem}
\label{lem:characterization_Hone}
    Let $f\in \Hone$. Then there exists a function $\zeta\in H^1(\Omega)$ such that
    \begin{equation}
        \label{eq:characterization_Hone}
       \begin{aligned}
           \|\nabla\zeta\|^2_{L^2(\Omega)}+ \frac{1}{{h}}&\|f-\zeta\|_{L^2(\partial\Omega)}^2+\frac{1}{{h}^2}\|f-\zeta\|^2_{L^2(\Omega)}\\
           &\lesssim \lvert f \rvert_{\Hone}^2 + \sum_{e\in\pThi} \frac{1}{|e|}\left\|\Pi_{0,e}[[f]]_e \right\|^2_{L^2(e)}
       \end{aligned}
    \end{equation}
    where $\Pi_{0,e}$ is the orthogonal projection operator from $L^2(e)$ onto $\Pzero(e)$, the space of constant functions on $e$.
\end{lem}
\begin{proof}
    See \cite[Proposition 2.7]{brenner2015piecewise}.
\end{proof}

We observe that the jump term appears again in the inequality, similar to the discrete Poincaré inequality for classical non-conforming elements discussed in Section \ref{subsec:discrete_poincare_brenner}. This observation leads us to concentrate on functions from the CR space $\CRspace$, where the jump term is eliminated. Consequently, we obtain the following result:

\begin{lem}
    \label{lem:CR_trace_inequality}
    Let $\omega_h\in\CRspace$, then the following estimates hold:
    \begin{equation}
        \label{eq:CR_trace_estimate1}
         \|\omega_h\|^2_{\Ltwoboundaryb} \lesssim 1+{h}^2 \lvert \omega_h \rvert_{\Hone}^2 + \left(\int_{\Omega} \omega_h\,dx\right)^2
    \end{equation}
    and
    \begin{equation}
        \label{eq:CR_trace_estimate2}
         \|\omega_h\|^2_{\Ltwoboundaryb} \lesssim  (1+{h}) \lvert \omega_h \rvert_{\Hone}^2 + \left(\int_{\Gamma} \omega_h\,ds\right)^2,
    \end{equation}
{where $\Gamma$ is combination of boundary faces that has a positive measure}, namely, $\Gamma=\bigcup_{i=1}^N e_i$ such that $e_i\in\pThb$ and $\{e_1,e_2,\cdots,e_N\}$ are different boundary faces.
\end{lem}
\begin{proof}
   For a given $\omega_h\in\CRspace$ and based on Lemma \ref{lem:characterization_Hone}, there exists a function $\zeta\in H^1(\Omega)$ satisfying the relationship described in \eqref{eq:characterization_Hone},

\begin{equation}
\label{eq:CR_Hone_characterization}
    \begin{aligned}
        \|\nabla\zeta\|^2_{L^2(\Omega)}+ \frac{1}{{h}}&\|\omega_h-\zeta\|_{L^2(\partial\Omega)}^2+\frac{1}{{h}^2}\|\omega_h-\zeta\|^2_{L^2(\Omega)}\lesssim \lvert \omega_h \rvert_{\Hone}^2 
    \end{aligned}
\end{equation}
as every interior face satisfies,
\begin{equation*}
    \int_e [[\omega_h]]_e\,ds=0.
\end{equation*}
Next, by decomposing $\omega_h$ into $(\omega_h-\zeta)$ and $\zeta$, we can derive an estimate for the trace of $\omega_h$ on $\pThb$ as follows:
\begin{equation}
\label{eq:trace_split}
    \begin{aligned}
        \|\omega_h\|^2_{\Ltwoboundaryb}&\lesssim \|\omega_h-\zeta\|^2_{\Ltwoboundaryb}+\|\zeta\|^2_{\Ltwoboundaryb}.
    \end{aligned}
\end{equation}
   The first part has already appeared in \eqref{eq:CR_Hone_characterization}. Regarding the second portion, given that $\zeta\in H^1(\Omega)$, the classical trace theorem and Poincaré inequality related to the mean-value are applicable, resulting in:
\begin{equation*}
    \begin{aligned}
        \|\zeta\|_{\Ltwoboundaryb}^2&\lesssim \|\nabla\zeta\|_{L^2(\Omega)}^2+\left(\int_{\Omega} \zeta\,dx\right)^2\\
        &\lesssim \|\nabla\zeta\|_{L^2(\Omega)}^2+\left[\int_{\Omega} (\zeta-\omega_h)\,dx\right]^2+\left(\int_{\Omega} \omega_h\,dx\right)^2\\
        &\lesssim \|\nabla\zeta\|_{L^2(\Omega)}^2+ \|\omega_h-\zeta\|^2_{\Ltwo}  +\left(\int_{\Omega} \omega_h\,dx\right)^2.
    \end{aligned}
\end{equation*}
The last line is deduced through employing the Cauchy-Schwarz inequality and acknowledging $\Omega$'s finite measure. Analogously, applying a Poincaré inequality relative to the boundary mean value provides:
\begin{equation*}
    \begin{aligned}
        \|\zeta\|_{\Ltwoboundaryb}^2&\lesssim \|\nabla\zeta\|_{L^2(\Omega)}^2+\left(\int_{\Gamma} \zeta\,ds\right)^2\\
        &\lesssim \|\nabla\zeta\|_{L^2(\Omega)}^2+\left[\int_{\partial\Omega} (\zeta-\omega_h)\,ds\right]^2+\left(\int_{\Gamma} \omega_h\,ds\right)^2\\
        &\lesssim \|\nabla\zeta\|_{L^2(\Omega)}^2+ \|\omega_h-\zeta\|^2_{L^2(\pThb)}  +\left(\int_{\Gamma} \omega_h\,ds\right)^2,
    \end{aligned}
\end{equation*}
due to $\Gamma$'s finite measure. Integrating these findings into \eqref{eq:trace_split} and associating it with \eqref{eq:CR_Hone_characterization} leads to:
\begin{align*}
    \|\omega_h\|^2_{\Ltwoboundaryb} &\lesssim \|\nabla\zeta\|_{L^2(\Omega)}^2+ \|\omega_h-\zeta\|^2_{L^2(\pThb)}  + \|\omega_h-\zeta\|^2_{\Ltwo}  +\left(\int_{\Omega} \omega_h\,dx\right)^2\\
    &\lesssim\left[1+{h}+{h}^2\right]\lvert \omega_h \rvert_{\Hone}^2 + \left(\int_{\Omega} \omega_h\,dx\right)^2\\
    &\lesssim\left[1+{h}^2\right]\lvert \omega_h \rvert_{\Hone}^2 + \left(\int_{\Omega} \omega_h\,dx\right)^2,
\end{align*}
and
\begin{align*}
    \|\omega_h\|^2_{\Ltwoboundaryb} &\lesssim \|\nabla\zeta\|_{L^2(\Omega)}^2+ \|\omega_h-\zeta\|^2_{\Ltwoboundary}  +\left(\int_{\Gamma} \omega_h\,ds\right)^2\\
    &\lesssim (1+{h}) \lvert \omega_h \rvert_{\Hone}^2 + \left(\int_{\Gamma} \omega_h\,ds\right)^2.
\end{align*}
We here yield the anticipated outcomes.

\end{proof}

\begin{rem}
   This finding indicates that when the average value of a CR element on the boundary can be evaluated, the classical trace theorem from the $H^1$ Sobolev space can be extended to the non-conforming CR space.
\end{rem}

With this insight, we now shift our focus to formulating a trace argument for hybridizable spaces, employing a methodology akin to that used in proving the Poincaré inequality. Our goal is to establish a bridge between $\uh$ and $\uhath$ by examining the discrepancies between their average values and their individual values, and subsequently mapping these boundary averages into CR spaces. This process leads to the following theorem:

\begin{thm}\label{thm:trace}
    Let $(\uh,\uhath)\in\Xhk$, then the following trace inequalities hold:
    \begin{equation}\label{eq:hybrid_trace_u}
        \|\uh\|_{\Ltwoboundaryb}^2\lesssim {h}\lvert \uh\rvert_{\Hone}^2+\|\uh-\uhath\|^2_{L^2(\pThb)}+(1+{h})\lvert \CRlifting(\uhataverage) \rvert_{\Hone}^2 + \left(\int_{\Gamma} \uhath\,ds\right)^2
    \end{equation}
    and 
    \begin{equation}\label{eq:hybrid_trace_uhat}
        \|\uhath\|_{\Ltwoboundaryb}^2\lesssim {h}\lvert \uh\rvert_{\Hone}^2+\|\uh-\uhath\|^2_{L^2(\pThb)}+(1+{h})\lvert \CRlifting(\uhataverage) \rvert_{\Hone}^2 + \left(\int_{\Gamma} \uhath\,ds\right)^2.
    \end{equation}
    Here $\Gamma $ has the same definition as Lemma \ref{lem:CR_trace_inequality}.
\end{thm}
\begin{proof}
We evaluate the estimate for $\|\uh\|_{\Ltwoboundaryb}$. This norm is derived from the trace of $\uh$ on $\pThb$. By the definition of $\pThb$, the boundary of the mesh can be expressed as $\pThb=\bigcup_{i=1}^{N_{h_K}}e_i$ for a given mesh, where $\pThb$ consists of $N_{h}$ distinct faces. Therefore, we can state:
\begin{equation*}
    \|\uh\|_{\Ltwoboundaryb}^2=\sum_{i=1}^{N_{h}} \|\uh\|_{L^2(e_i)}^2.
\end{equation*}
{For each boundary face $e_i$, let $K_i$ denote the element it belongs to. Note that $K_i$ may not be unique; in fact, for each interior face $e_i$, there exist two adjacent elements, denoted as $K_i^+$ and $K_i^-$. However, for simplicity, we refer to them collectively as $K_i$ in this context.} Then we can decompose $\uh$ into two parts and achieve:
\begin{equation}\label{eq:uh_trace_split}
    \|\uh\|_{L^2(e_i)}^2\lesssim \|\uh-\uhaverageEdgei\|_{L^2(e_i)}^2 + \|\uhaverageEdgei\|^2_{L^2(e_i)}
\end{equation}
where $\uhaverageEdgei$ is the average value of $\uh$ on face $e_i$, defined by:
\begin{equation*}
    \uhaverageEdgei=\frac{1}{|e_i|}\int_{e_i}\uh\,ds.
\end{equation*}
For the first term, utilizing Lemma \ref{lemma:trace_inequality_simplex} and Lemma \ref{lem:Poincare_Frederich_discrete_boundary_average}, which are the discrete trace inequality and discrete Poincaré inequality in a simplex, we can establish bound as follows:
\begin{equation}\label{eq:uh_trace_poincare}
    \|\uh-\uhaverageEdgei\|_{L^2(e_i)}^2\lesssim \frac{1}{{h_{K_i}}}\|\uh-\uhaverageEdgei\|^2_{L^2(K_i)}\lesssim {h_{K_i}}\|\nabla\uh\|^2_{L^2(K_i)} {\lesssim h \|\nabla\uh\|^2_{L^2(K_i)} },
\end{equation}
Given that $\uhaverageEdgei$ is a constant, we can estimate the second term as follows:
\begin{equation*}
    \begin{aligned}
        (\uhaverageEdgei)^2&=\frac{1}{|e_i|^2}\left(\int_{e_i}\uh\,ds\right)^2\\
        &\lesssim  \frac{1}{|e_i|^2} \left[\int_{e_i}(\uh-\uhath)\,ds\right]^2 + \frac{1}{|e_i|^2}\left(\int_{e_i}\uhath\,ds\right)^2\\
        &\lesssim \frac{1}{|e_i|}\|\uh-\uhath\|^2_{L^2(e_i)}+(\uhathaverageEdgei)^2,
    \end{aligned}
\end{equation*}
where $\uhathaverageEdgei$ is detailed in \eqref{eq:side_average_hat}. Consequently,
\begin{equation*}
    \|\uhaverageEdgei\|^2_{L^2(e_i)} = |e_i|(\uhaverageEdgei)^2\lesssim \|\uh-\uhath\|^2_{L^2(e_i)} + \|\uhathaverageEdgei\|^2_{L^2(e_i)}.
\end{equation*}
Following Lemma \ref{lem:CRLifting_estimate} and especially, estimate \eqref{eq:uhat_average_lifting_comparison_estimate}, it's evident that:
\begin{equation*}
    \|\uhaverageEdgei\|^2_{L^2(e_i)} \lesssim \|\uh-\uhath\|^2_{L^2(e_i)} + \|\CRlifting(\uhataverage)\|^2_{L^2(e_i)},
\end{equation*}
with $\uhataverage$ specified in \eqref{eq:uhat_average_def}. Merging these results leads to:
\begin{equation}\label{eq:trace_uh_face_estimate}
    \|\uh\|_{L^2(e_i)}^2\lesssim {h}\|\nabla\uh\|^2_{L^2(K_i)}+\|\uh-\uhath\|^2_{L^2(e_i)} + \|\CRlifting(\uhataverage)\|^2_{L^2(e_i)}.
\end{equation}
Summarizing these results gives:
\begin{equation}
\label{eq:trace_uh_sum_estimate}
    \|\uh\|_{\Ltwoboundaryb}^2=\sum_{i=1}^{N_h} \|\uh\|_{L^2(e_i)}^2\lesssim {h}\lvert \uh\rvert_{\Hone}^2+\|\uh-\uhath\|^2_{L^2(\pThb)}+\|\CRlifting(\uhataverage)\|_{L^2(\pThb)}^2.
\end{equation}
By \eqref{eq:CR_trace_estimate2} from Lemma \ref{lem:CR_trace_inequality} and the definition of CR lifting operator, the trace of $\CRlifting(\uhataverage)$ can be bounded as
\begin{equation}
    \label{eq:CR_lifting_trace_estimate}
      \begin{aligned}
          \|\CRlifting(\uhataverage)\|^2_{\Ltwoboundaryb} &\lesssim (1+{h}) \lvert \CRlifting(\uhathaverageEdge) \rvert_{\Hone}^2 + \left(\int_{\Gamma} \CRlifting(\uhataverage)\,ds\right)^2\\
      &=(1+{h})\lvert \CRlifting(\uhathaverageEdge) \rvert_{\Hone}^2 + \left(\int_{\Gamma} \uhataverage\,ds\right)^2\\
    &=(1+{h})\lvert \CRlifting(\uhathaverageEdge) \rvert_{\Hone}^2 + \left(\int_{\Gamma} \uhath\,ds\right)^2.
      \end{aligned}
\end{equation}
Inserting \eqref{eq:uh_trace_poincare}, \eqref{eq:trace_uh_sum_estimate} and \eqref{eq:CR_lifting_trace_estimate} into \eqref{eq:uh_trace_split}, we have shown the estimate \eqref{eq:hybrid_trace_u}.
To obtain $\eqref{eq:hybrid_trace_uhat}$, we simply need to notice the fact that 
\begin{equation*}
    \|\uhath\|_{\Ltwoboundaryb}^2\lesssim \|\uh\|_{\Ltwoboundaryb}^2 + \|\uhath-\uh\|_{\Ltwoboundaryb}^2
\end{equation*}
and then the desired result follows.
\end{proof}

\begin{rem}
    In the estimates \eqref{eq:hybrid_trace_u} and \eqref{eq:hybrid_trace_uhat}, the term $\left(\int_{\Gamma} \uhath\,ds\right)^2$ can be replaced with $\left(\int_{\Gamma} \uh\,ds\right)^2$. This adjustment is viable as the term $\|\uhath-\uh\|_{\Ltwoboundaryb}^2$ is involved in the estimate. Moreover, it's also practical to exchange this term with the integral of $\CRlifting(\uhataverage)$ over the entire domain $\Omega$, opting for \eqref{eq:CR_trace_estimate1} over \eqref{eq:CR_trace_estimate2} to establish the trace estimate for $\|\CRlifting(\uhataverage)\|_{L^2(\pThb)}^2$.
\end{rem}

\section{Application to HDG formulation}\label{sec:HDG}

In this section, we will discuss how Poincar\'e inequalities and trace inequalities developed above can actually benefit in analysis for problems set up by HDG method. In particular, we will use these tools to obtain uniform energy estimates for the solutions for second-order elliptic equations in dependent of mesh size $h_K$ with a minimal regularity assumption. For sake of simplicity, we will only consider Poisson equation as a model problem and the analysis can be extended to more general case of second-order elliptic equations
easily.

\subsection{Boundary Lifting Operator}\label{subsec:boundary_lifting}
As a start, we introduce the definition of a discrete gradient operator, which will be called a lifting operator in the following. This operator is designed to approximate the distributional gradient which is also a common methodology in analyzing discontinuous schemes presented in related works \cite{buffa2009compact,di2010discrete,kikuchi2012rellich,kirk2022numerical,kirk2023convergence,shen2017hybridizable,ten2006discontinuous}. The cornerstone of this discrete gradient operator lies in a critical observation regarding the nature of functions within $\Xhk$ \cite{buffa2009compact}. Specifically, it is noted that these functions exhibit discontinuities, which in turn implies that their distributional gradient is influenced by the difference of $\uh$ and $\uhath$ on the interfaces of elements. 

In each element $K$, we introduce a local lifting operator $\gradientKlifting: L^2( \partial K) \to \pmb{\mathcal{P}}^k(K)$, inspired by the previously discussed contents. This operator transforms a function $\hat{\mu}$, defined on $\partial K$, into a vector-valued piecewise polynomial function. Specifically, for each function $\hat{\mu}\in L^2( \partial K)$, we define $\gradientKlifting (\hat{\mu})$ as follows:
\begin{equation}\label{eq:lifting_operator}
    \int_K\gradientKlifting (\hat{\mu})\cdot \pmb{\omega_h}\,dx = \int_{\partial K} \hat{\mu}\, \pmb{\omega_h}\cdot\pmb{n}\,ds
\end{equation}
for any $\pmb{\omega_h}\in \pmb{\mathcal{P}}^k(K)$. The global lifting operator $\gradientlifting: \Ltwoboundary\to \Vhk$ is then defined through the restriction to each element, such that $\gradientlifting(\hat{\mu})|_K = \gradientKlifting(\hat{\mu}|_K)$ for every $\hat{\mu}\in\Ltwoboundary$. Hence, it should comply with the form:
\begin{equation}
    \int_\Omega \gradientlifting (\hat{\mu})\cdot \pmb{\omega_h}\,dx = \sum_{K\in\Th} \int_K \gradientlifting (\hat{\mu})\cdot \pmb{\omega_h}\,dx = \sum_K \int_{\partial K} \hat{\mu}\, \pmb{\omega_h}\cdot\pmb{n}\,ds, \quad\quad \forall \pmb{\omega_h}\in\Vhk.
\end{equation}

The following lemma provides a local estimate for this lifting operator, emerging as a direct consequence of the discrete trace inequality Lemma \ref{lemma:trace_inequality_simplex}.
\begin{lem}
\label{lem:lifting_operator}
For every $\hat{\mu}\in L^2(\partial K)$,
\begin{equation*}
    \|\gradientKlifting(\hat{\mu})\|^2_{\pmb{L}^2(K)}\lesssim \frac{1}{h_K}\|\hat{\mu}\|_{L^2(\partial K)}^2.
\end{equation*}
\end{lem}
\begin{proof}
    Let $\pmb{\omega_h}=\gradientKlifting(\hat{\mu})$ in equation \eqref{eq:lifting_operator}, by Cauchy-Schwarz inequality we have
    \begin{equation*}
        \|\gradientKlifting(\hat{\mu})\|^2_{\pmb{L}^2(K)}=\int \hat{\mu}\, \gradientKlifting(\hat{\mu})\cdot\pmb{n}\,ds\leq \|\hat{\mu}\|_{L^2(\partial K)}\|\gradientKlifting(\hat{\mu})\cdot\pmb{n}\|_{L^2(\partial K)}.
    \end{equation*}
    By Lemma \ref{lemma:trace_inequality_simplex}, combining with Remark \ref{rem:trace_inequality_vector} and Remark \ref{rem:trace_inequality_hK}, we can conclude that
    \begin{equation*}
        \|\gradientKlifting(\hat{\mu})\cdot\pmb{n}\|_{L^2(\partial K)}\lesssim \frac{1}{h_K^{\frac{1}{2}}}\|\gradientKlifting(\hat{\mu})\|_{L^2(K)}.
    \end{equation*}
    Combining these two formulas and the claim of this lemma follows.
\end{proof}

This result can be extended to be an estimate for the global lifting operator by direct addition of the part in each element, which is summarized as follows.

\begin{cor}
For $\hat{\mu}
\in L^2(\partial\Th)$,  
\begin{equation*}
    \|\gradientlifting(\hat{\mu})\|^2_{\Ltwovector}\lesssim \frac{1}{{h}}\|\hat{\mu}\|_{\Ltwoboundary}^2.
\end{equation*}
\end{cor}

\subsection{Problem Setup}

We will briefly outline the HDG method's formulation and structure for the Poisson problem, then explore how the Poincaré inequality and the trace inequality we derived can be used for stability analysis towards it. This contrasts with \cite{jiang2023stability}, where a translation argument was employed for deducing stability.

In this discussion, we address the Poisson equation with mixed boundary conditions as a model problem. Other types of boundary conditions can also be accommodated within this framework. The strong form of the Poisson equation in $\Omega$ is given by:

\begin{equation}\label{eq:Poisson_equation}
\left\{
    \begin{array}{ll}
        -\Delta u = f & \text{in } \Omega, \\
       u = u_D & \text{on } \Gamma_D,\\
      \nabla u\cdot\pmb{n} = u_N & \text{on } \Gamma_N.
    \end{array}
\right.
\end{equation}
Here, $\partial\Omega=\bar{\Gamma}_D\bigcup \bar{\Gamma}_N$ and $\Gamma_D\bigcap\Gamma_N=\emptyset$, with $f$ serving as the source term.

A mixed formulation is introduced by defining $\pmb{p}=-\nabla u$, allowing the system to be reformulated as:

\begin{equation}\label{eq:Poisson_equation_mixed}
\left\{
    \begin{array}{ll}
        \nabla \cdot \pmb{p} = f & \text{in } \Omega, \\
         \pmb{p} + \nabla u = 0 & \text{in } \Omega, \\
   u=u_D & \text{on } \Gamma_D,\\
    \pmb{p}\cdot \pmb{n} = -u_N &  \text{on } \Gamma_N.
    \end{array}
\right.
\end{equation}
In situations where the solution possesses sufficient regularity, these two formulations are equivalent. To solve this problem numerically, the domain is partitioned into a mesh $\Th$, and we will continue employing the notations introduced in Section \ref{subsec:notation}. Upon establishing a mesh, we adhere to the standard HDG formulation for second-order elliptic equations as documented in \cite{cockburn2004characterization,cockburn2009unified,cockburn2009superconvergent,sevilla2016tutorial} to devise the scheme. Specifically, within each element $K$, our objective is to find $(\pmb{p}_h,\uh)\in \Vhk\times \Uhk$ fulfilling:
\begin{equation}
    \label{eq:HDG_formulation_p_nabla_u}
    (\ph,\qh)_K = (\uh,\nabla\cdot\qh)_K-\langle \uhath,\qh\cdot\pmb{n}  \rangle _{\Kboundary},
\end{equation}
and
\begin{equation}\label{eq:HDG_formulation_p_f}
    -(\ph, \nabla\vh)_K+\langle \hat{\pmb{p}}_h\cdot\pmb{n}, v \rangle_{\Kboundary} = (f,v)_K,
\end{equation}
for every test function pair $(\qh,\vh)\in\Vhk\times\Uhk$. The Dirichlet boundary condition is imposed as \cite{cockburn2010projection}:
\begin{equation}
    \label{eq:HDG_Dirichlet}
    \langle \uhath,\hat{\mu}\rangle_{\Gamma_D} = \langle u_D,\hat{\mu}\rangle_{\Gamma_D},
\end{equation}
for all $\hat{\mu}\in\Fhk$. Numerical traces of the fluxes in the HDG scheme are typically chosen as \cite{cockburn2009unified,nguyen2009implicit,nguyen2009implicitnonlinear,nguyen2010hybridizable,nguyen2011implicit}:
\begin{equation}
    \label{eq:flux_trace}
    \hat{\pmb{p}}_h\cdot\pmb{n} = {\pmb{p}_h}\cdot\pmb{n}+\tau(\uh-\uhath),
\end{equation}
where $\tau$ is a stabilization function significantly affecting the scheme's effectiveness and accuracy. Numerous studies have been dedicated to this selection, for instance, \cite{cockburn2010projection,cockburn2009superconvergent,kirby2012cg} and references therein. It is noted that choosing $\tau$ as a constant on a simplicial mesh ensures optimal convergence order. However, selecting the stabilization function to be of order $O(\frac{1}{h_K})$ results in a loss of one convergence order in both the locally post-processed approximation to the scalar variable and the approximate to the gradient. Yet, conducting stability analysis for the constant case presents more challenges from a traditional standpoint. We aim to focus on this scenario using the newly developed tools above.

Once we have established the local problems, a global problem can be formulated to determine $\uhath$, considering the behavior of the numerical fluxes as outlined in \cite{sevilla2016tutorial}:
\begin{equation}\label{eq:HDG_formulation_flux_jump}
    \left\langle \hat{\pmb{p}}_h\cdot\pmb{n},\hat{\mu}\right\rangle_{\pThi}+ \left\langle \hat{\pmb{p}}_h\cdot\pmb{n},\hat{\mu} \right\rangle_{\Gamma_N} = -\left\langle \hat{\mu},u_N\right\rangle_{\Gamma_N}.
\end{equation}
This equation also implements the Neumann boundary.

Merging \eqref{eq:HDG_formulation_p_nabla_u}-\eqref{eq:HDG_formulation_flux_jump} leads to summarizing the HDG formulation as the following task: Finding $(\ph,\uh,\uhath)\in \Vhk\times\Uhk\times\Fhk$ such that:
\begin{subequations}\label{eq:HDG_scheme}
\begin{equation}
     (\ph,\qh)_{\Th} = (\uh,\nabla\cdot\qh)_{\Th}-\langle \uhath,\qh\cdot\pmb{n}  \rangle _{\partial\Th}, \label{eq:HDG_scheme1a}
\end{equation}
\begin{equation}
   \tau \left\langle \uh-\uhath, \vh-\hat{v}_h \right\rangle_{\partial\Th}+ (\nabla\cdot \ph, \vh)_{\Th}-\left\langle  \ph\cdot\pmb{n}, \hat{v}_h\right\rangle_{\partial\Th} = (f,v)_{\Th} + \langle u_N,\hat{v}_h\rangle_{\Gamma_N} ,\label{eq:HDG_scheme1b}
\end{equation}
\begin{equation}
  \langle\uhath,\hat{v}_h\rangle_{\Gamma_D} = \langle u_D,\hat{v}_h\rangle_{\Gamma_D} \label{eq:HDG_scheme1c}
\end{equation}
\end{subequations}
for all $(\qh,\vh,\hat{v}_h)\in \Vhk\times\Uhk\times\Fhk$. For simplicity, $\tau$ will be assumed as a fixed positive constant throughout. In a more general context, $\tau$ could be regarded as a function defined over the skeleton with a positive lower bound. The ensuing analysis would remain applicable to such a scenario. It's also assumed, without loss of generality, that the mesh properly decomposes $\Gamma_D$ and $\Gamma_N$, meaning each can be expressed as a union of non-disjoint boundary faces. The well-posedness of this scheme has been presented in several studies, such as \cite{cockburn2010hybridizable}. The subsequent energy-type argument is standard and directly follows from the HDG scheme.

\begin{lem}\label{lem:energy_HDG}
    The numerical solution $(\ph,\uh,\uhath)\in \Vhk\times\Uhk\times\Fhk$, solving \eqref{eq:HDG_scheme}, satisfies:
    \begin{equation}
        \label{eq:hdg_energy}
        \|\ph\|^2_{\Ltwovector}+\tau\|\uh-\uhath\|^2_{\Ltwoboundary}=(f,\uh)_{\Th} + \langle u_N,\uhath\rangle_{\Gamma_N},
    \end{equation}
    and
    \begin{equation}
        \label{eq:Dirichlet_estimate}
        \|\uhath\|_{L^2(\Gamma_D)}^2 \leq \|u_D\|_{L^2(\Gamma_D)}^2.
    \end{equation}
\end{lem}

\begin{proof}
    By setting $\qh=\ph, \vh=\uh$, and $\hat{v}_h=\uhath$ in \eqref{eq:HDG_scheme1a} and \eqref{eq:HDG_scheme1b}, equation \eqref{eq:hdg_energy} is obtained, whereas \eqref{eq:Dirichlet_estimate} follows from the Cauchy-Schwarz inequality by choosing $\hat{v}_h=\uhath$ in \eqref{eq:HDG_scheme1c}.
\end{proof}

It's important to note that in the HDG scheme, energy estimation is conducted concerning $\ph$, rather than $\gradienth\uh$ or $\nabla\CRlifting(\uhataverage)$. Consequently, the direct application of the analytical techniques we've previously discussed is not feasible without establishing a proper connection between $\ph$ and these elements. Identifying this relationship will be the primary objective in the following discussions.

\subsection{Poincar\'e and Trace Inequalities for HDG}

Building on the previous discussion, our goal here is to elucidate the relationships between $\|\uh\|_{\Ltwo}$, $\|\uhath\|_{\Ltwoboundary}$, and $\|\ph\|_{\Ltwovector}$, which are crucial to the analysis of HDG scheme. We recognize that \eqref{eq:HDG_scheme1a} uniquely defines $\ph$ in terms of $(\uh,\uhath)$, serving as the bridge for us to link $\gradienth \uh$ with $\ph$, and similarly, to connect $\gradienth\CRlifting(\uhataverage)$ with $\ph$. These connections will lay the groundwork for adapting the Poincaré and trace inequalities, previously established for hybridizable spaces, to the specific context of the HDG scheme.

We initiate our analysis by examining $\nabla u$ within each element $K$, which yields the subsequent result:

\begin{lem}\label{lem:gradient_u_estimate}
Given $(\uh,\uhath)\in\Xhk$, and assuming $\ph$ satisfies \eqref{eq:HDG_scheme1a}, then:
    \begin{equation}
        \|\nabla \uh\|^2_{L^2(K)}\lesssim \|\ph\|^2_{\pmb{L}^2(K)}+\frac{1}{h_K}\|u-\uhath\|^2_{L^2(\Kboundary)}.
    \end{equation}
\end{lem}

\begin{proof}
Using integration by parts, the equation \eqref{eq:HDG_scheme1a} can be reformulated as:
\begin{equation*}
      (\ph,\qh)_K = -(\nabla\uh,\qh)_K+\langle \uh-\uhath,\qh\cdot\pmb{n}  \rangle _{\Kboundary}.
\end{equation*}
Referring to the definition of the lifting operator in \eqref{eq:lifting_operator}, we find an equivalent expression:
\begin{equation*}
    (\nabla \uh, \qh)_K=(-\ph+\gradientKlifting(\uh-\uhath), \qh)_K.
\end{equation*}
This holds true for all $\qh\in\Pk(K)$. Given that $\nabla u, \ph, \gradientKlifting(u-\uhath)$ are all functions in $\Pk(K)$, it follows that:
\begin{equation*}
    \nabla \uh=-\ph+\gradientKlifting(\uh-\uhath).
\end{equation*}
Thus, we deduce that:
\begin{equation*}
    \|\nabla \uh\|^2_{L^2(K)}=\|-\ph+\gradientKlifting(\uh-\uhath)\|^2_{L^2(K)} \lesssim \|\ph\|^2_{L^2(K)}+\|\gradientKlifting(\uh-\uhath)\|_{L^2(K)}^2.
\end{equation*}
The conclusion is easy to drawn by Lemma \ref{lem:lifting_operator}.
\end{proof}

Another estimate derived from \eqref{eq:HDG_scheme1a} is the quantitative relationship between $\ph$ and $\CRlifting(\uhataverage)$.

\begin{lem}
    \label{lem:Crouieux_Raviart_lifting_estimate}
    Given $(\uh,\uhath)\in\Xhk$ and assuming $\ph$ satisfies \eqref{eq:HDG_scheme1a}, then:
    \begin{equation}
        \|\nabla\CRlifting(\uhataverage)\|_{L^2(K)}\leq \|\ph\|_{\pmb{L}^2(K)}
    \end{equation}
    in each element $K$.
\end{lem}
\begin{proof}
   Selecting $\qh = \nabla\CRlifting(\uhataverage)$ in \eqref{eq:HDG_scheme1a} and noting that $\CRlifting(\uhataverage)\in \Pone(K)$ implies $\nabla\CRlifting(\uhataverage)$ is a constant vector in $K$. Hence,
   \begin{equation*}
       \nabla\cdot \left(\nabla \CRlifting(\uhataverage)\right)=0.
   \end{equation*}
   Consequently, \eqref{eq:HDG_scheme1a} simplifies to:
   \begin{align*}
       (\ph,\nabla\CRlifting(\uhataverage))_K = -\langle\uhath,\nabla\CRlifting(\uhataverage)\cdot\pmb{n}\rangle_{\partial K}.
   \end{align*}
   Given that $\nabla\CRlifting(\uhataverage)\cdot\pmb{n}$ is also constant, it follows that:
   \begin{align*}
       (\ph,\nabla\CRlifting(\uhataverage))_K = -\langle\uhataverage,\nabla\CRlifting(\uhataverage)\cdot\pmb{n}\rangle_{\partial K} = -\langle\CRlifting(\uhataverage),\nabla\CRlifting(\uhataverage)\cdot\pmb{n}\rangle_{\partial K}= - \|\nabla\CRlifting(\uhataverage)\|^2_{\pmb{L}^2(K)},
   \end{align*}
   with the final equality due to integration by parts. 
    Applying the Cauchy-Schwarz inequality to the left-hand side yields the lemma's statement.
\end{proof}

We immediately obtain the following theorems in terms of $\ph$ which are variants of Theorem \ref{thm:Poincare} and Theorem \ref{thm:trace}. The proof follows directly by incorporating Lemma \ref{lem:gradient_u_estimate} and Lemma \ref{lem:Crouieux_Raviart_lifting_estimate} into these theorems.

\begin{thm}\label{thm:Poincare_ph}
    Let $(\uh,\uhath)\in\Xhk$, and $\ph$ is obtained solved through \eqref{eq:HDG_scheme1a}, then the following Poincar\'e inequalities hold:
    \begin{equation}\label{eq:poincare_mean1_ph}
          \|\uh\|^2_{\Ltwo}\lesssim (1+{h}^2)\|\ph\|_{\Ltwovector}^2 +{h} \| \uh-\uhath\|^2_{L^2(\partial\Th)}+  \left(\int_\Omega\CRlifting(\uhataverage)\,dx \right)^2
    \end{equation}
    and
    \begin{equation}\label{eq:poincare_boundary_mean_ph}
          \|\uh\|^2_{\Ltwo}\lesssim (1+{h}^2)\|\ph\|_{\Ltwovector}^2+{h} \| \uh-\uhath\|^2_{L^2(\partial\Th)} + \left(\int_\Gamma\uhath\,ds \right)^2
    \end{equation}
   { where $\Gamma$ is combination of boundary faces that has a positive measure}, namely, $\Gamma=\bigcup_{i=1}^N e_i$ such that $e_i\in\pThb$ and $\{e_1,e_2,\cdots,e_N\}$ are different boundary faces. In addition, the following variant of \eqref{eq:poincare_mean1_ph} expressing in term of integral of $\uh$ also holds:
    \begin{equation}\label{eq:poincare_mean2_ph}
          \|\uh\|^2_{\Ltwo}\lesssim (1+{h}^2)\|\ph\|_{\Ltwovector}^2 +{h}\| \uh-\uhath\|^2_{L^2(\partial\Th)}+ \left(\int_\Omega\uh\,dx \right)^2.
    \end{equation}
\end{thm}

\begin{thm}\label{thm:trace_p}
    Let $(\uh,\uhath)\in\Xhk$, and $\ph$ is obtained solved through \eqref{eq:HDG_scheme1a}, then the following trace inequalities hold:
    \begin{equation}\label{eq:hybrid_trace_u_ph}
        \|\uh\|_{\Ltwoboundaryb}^2\lesssim (1+{h})\|\ph\|^2_{\Ltwovector}+\|\uh-\uhath\|^2_{L^2(\partial\Th)}+\left(\int_{\Gamma} \uhath\,ds\right)^2
    \end{equation}
    and 
    \begin{equation}\label{eq:hybrid_trace_uhat_ph}
        \|\uhath\|_{\Ltwoboundaryb}^2\lesssim (1+{h})\|\ph\|^2_{\Ltwovector}+\|\uh-\uhath\|^2_{L^2(\partial \Th)}+ \left(\int_{\Gamma} \uhath\,ds\right)^2,
    \end{equation}
     {where $\Gamma$ is combination of boundary faces that has a positive measure}, namely, $\Gamma=\bigcup_{i=1}^N e_i$ such that $e_i\in\pThb$ and $\{e_1,e_2,\cdots,e_N\}$ are different boundary faces.
\end{thm}

\subsection{Stability Analysis}

To the end, we address an application of the mathematical tools developed in this study to examine the stability of the HDG formulation \eqref{eq:HDG_scheme} for the mixed boundary Poisson equation, specifically when the stabilization term is chosen to be a constant.

A crucial inquiry we pursue is whether it is possible to obtain an energy estimate for the HDG solution that does not depend on the mesh size $h_K$, without assuming additional regularity for the solution to the Poisson equation beyond the minimum requirements for the data $f$, $u_D$, and $u_N$. Existing research typically assumes the existence of a solution in a "strong" sense (with varying interpretations of "strong") and employs a projection-based analysis to verify the numerical solution's stability \cite{cockburn2010projection}. In a previous work \cite{jiang2023stability}, we initially proved the stability of the numerical solution with only minimal regularity, utilizing a translation argument. In this paper, our objective is to demonstrate the same result using a distinct approach, leveraging the mathematical instruments we have developed and making it easier to adapt to other type of problems.

More specifically, we seek to determine: for a given mesh $\Th$, can $\|\uh\|_{\Ltwo}$ and $\|\ph\|_{\Ltwovector}$ be uniformly bounded by a constant solely dependent on the input data and independent of the mesh size? This question is affirmatively addressed in the subsequent theorem:

\begin{thm}
    \label{thm:HDG_stability} Let $(\ph,\uh,\uhath)\in \Vhk\times\Uhk\times\Fhk$ solve \eqref{eq:HDG_scheme} with $f\in L^2(\Omega), u_D\in L^2(\Gamma_D), u_N\in L^2(\Gamma_N)$, then the following estimate hold:
    \begin{equation}
        \label{eq:HDG_energy_bound}
        \|\ph\|^2_{\Ltwovector}+\tau\|\uh-\uhath\|^2_{\Ltwoboundary}+\|\uhath\|_{L^2(\Gamma_D)}^2\leq C(f,u_D,u_N)
    \end{equation}
    where the constant $C(f,u_D,u_N)$ depends on $\|f\|_{L^2(\Omega)}, \|u_D\|_{L^2(\Gamma_D)}$, $\|u_N\|_{\Gamma_N}$ and the domain but independent of the mesh.
\end{thm}
\begin{proof}
    Recall Lemma \ref{lem:energy_HDG}, the following identity hold:
      \begin{equation}\label{eq:HDG_energy_theorem_use}
      \|\ph\|^2_{\Ltwovector}+\tau\|\uh-\uhath\|^2_{\Ltwoboundary}=(f,\uh)_{\Th} + \langle u_N,\uhath\rangle_{\Gamma_N}.
    \end{equation}
   
    We firstly assume that $\Gamma_D$ has a positive measure. In this case, the only thing we need to do is to bound $\|\uh\|_{\Ltwo}$ and $\|\uhath\|_{L^2(\Gamma_N)}$ as
    \begin{equation*}
        \left\lvert(f,\uh)_{\Th} + \langle u_N,\uhath\rangle_{\partial\Th}\right\rvert \leq \|f\|_{L^2(\Omega)}\,\|\uh\|_{\Ltwo} + \|u_N\|_{L^2(\Gamma_N)}\,\|\uhath\|_{L^2(\Gamma_N)}.    \end{equation*}
Choosing $\Gamma$ in the Poincar\'e inequality \eqref{eq:poincare_boundary_mean_ph} and the trace inequality \eqref{eq:hybrid_trace_uhat_ph} both to be $\Gamma_D$. As
\begin{equation*}
    \left(\int_{\Gamma_D} \uhath\,ds\right)^2\leq \left|\Gamma_D\right|\,\left(\int_{\Gamma_D}(\uhath)^2\,ds\right)= \left|\Gamma_D\right|\,\|\uhath\|_{L^2(\Gamma_D)}^2,
\end{equation*}
we can rewrite \eqref{eq:poincare_boundary_mean_ph} and \eqref{eq:hybrid_trace_uhat_ph} as 
\begin{equation}
    \begin{aligned}
        \|\uh\|^2_{\Ltwo}&\lesssim (1+{h}^2)\|\ph\|_{\Ltwovector}^2+{h} \| \uh-\uhath\|^2_{L^2(\partial\Th)} + \left(\int_{\Gamma_D}\uhath\,ds \right)^2\\
      &  \lesssim \|\ph\|_{\Ltwovector}^2+ \tau\| \uh-\uhath\|^2_{L^2(\partial\Th)} +\|\uhath\|_{L^2(\Gamma_D)}^2
    \end{aligned}
\end{equation}
and
\begin{equation}
    \begin{aligned}
        \|\uhath\|_{\Ltwoboundaryb}^2&\lesssim (1+{h})\|\ph\|^2_{\Ltwovector}+\|\uh-\uhath\|^2_{L^2(\pThb)}+ \left(\int_{\Gamma} \uhath\,ds\right)^2\\
        &\lesssim \|\ph\|^2_{\Ltwovector}+\tau\|\uh-\uhath\|^2_{L^2(\partial\Th)}+ \|\uhath\|_{L^2(\Gamma_D)}^2.
    \end{aligned}
\end{equation}
Using these estimates, together with $\eqref{eq:HDG_energy_theorem_use}$ and \eqref{eq:HDG_scheme1c} with $\hat{v}_h =\uhath$, we get
\begin{equation*}
    \begin{aligned} &\quad\,\,\|\ph\|^2_{\Ltwovector}+\tau\|\uh-\uhath\|^2_{\Ltwoboundary}+\|\uhath\|_{L^2(\Gamma_D)}^2\\
    &=(f,\uh)_{\Th} + \langle u_N,\uhath\rangle_{\Gamma_N} + \langle u_D,\uhath\rangle_{\Gamma_D}\\
    &\lesssim \left(\|f\|_{L^2(\Omega)}+\|u_N\|_{L^2(\Gamma_N)}+\|u_D\|_{L^2(\Gamma_D)}\right)\left(\|\ph\|^2_{\Ltwovector}+\tau\|\uh-\uhath\|^2_{\Ltwoboundary}+\|\uhath\|_{L^2(\Gamma_D)}^2 \right)^{\frac{1}{2}}.
    \end{aligned}
\end{equation*}
It gives the desired estimate \eqref{eq:HDG_energy_bound}.

When $|\Gamma_D|=0$, the scenario simplifies to a pure Neumann problem. According to classical elliptic theory, the solution to a Neumann problem is unique up to a constant \cite{evans2022partial}. Hence, to ensure uniqueness in the numerical scheme, one could introduce an additional condition such as:
\begin{equation*}
    \int_{\partial \Omega} \uhath\,ds=0\,\,\,\,\quad 
    \text{or }\,\,\,\,\quad 
    \int_\Omega \uh\,dx = 0.
\end{equation*}
Incorporating this condition into the numerical scheme eliminates the arbitrary constant, thereby securing a unique solution. The analysis procedure previously discussed remains applicable, as the Poincaré and trace inequalities would work to derive \eqref{eq:HDG_energy_bound}. The details are omitted here to avoid redundancy, as it repeats the earlier process.

\end{proof}

\section*{\textbf{Acknowledgements}}
We are grateful for the valuable discussions with Guosheng Fu, Jay Gopalakrishnan, Jiannan Jiang, and Noel Walkington. We also extend our thanks to the National Science Foundation (grant number: NSF DMS-1929284) for their support of this research.

\bibliographystyle{abbrv}
\bibliography{reference}

\end{document}